\documentclass[10pt,reqno, amssymb]{amsart}

\usepackage{mathtools} 

\usepackage[margin=0.9in]{geometry}

\usepackage{amsmath}
\usepackage{amssymb}
\usepackage{amsthm}
\usepackage{newlfont}
\usepackage{graphicx}

\newtheorem{thm}{Theorem}[section]

\newtheorem{lem}[thm]{Lemma}
\newtheorem{prop}[thm]{Proposition}

\theoremstyle{definition}

\theoremstyle{remark}
\newtheorem{rem}[thm]{Remark}
\newtheorem*{rem*}{Remark}
\numberwithin{equation}{section}

\newcommand{\ed}{\end {document}}

\title[Critical space]
{Optimal Gevrey regularity for supercritical quasi-geostrophic equations}

\author[D. Li]{Dong Li}
\address[D. Li]{Department of Mathematics, The University of Hong Kong, Hong Kong, China}
\email{mathdl@hku.hk}

\begin{document}
\begin{abstract}
We consider the two dimensional surface quasi-geostrophic equations with  super-critical 
dissipation. For large initial data in critical Sobolev and Besov spaces, we prove optimal Gevrey regularity
endowed with the same decay exponent as the linear part. This settles several  open problems
in \cite{Bis14, BMS15}.
\end{abstract}
\maketitle
\section{Introduction}
We consider the following two-dimensional dissipative surface quasi-geostrophic equation:
\begin{align} \label{1}
\begin{cases}
\partial_t \theta +(u \cdot \nabla) \theta +\nu D^{\gamma} \theta=0, 
\quad (t,x) \in (0,\infty) \times \mathbb R^2; \\
u= R^{\perp} \theta = (-R_2 \theta, R_1 \theta), \quad R_j =D^{-1} \partial_j;  \\
\theta (0,x) = \theta_0 (x),\quad x \in \mathbb R^2, 
\end{cases}
\end{align}
where  $\nu\ge 0$, $0<\gamma\le 2$,  $D=(-\Delta)^{\frac 12}$, $D^{\gamma}= (-\Delta)^{\frac{\gamma}2}$, 
and more generally the fractional operator $D^s =(-\Delta)^{\frac s 2}$ corresponds
to the Fourier multiplier $|\xi|^s$, i.e.
$
\widehat {D^s f }(\xi) = |\xi|^s \widehat{f} (\xi)
$
whenever it is suitably defined under certain regularity assumptions on $f$.  The 
scalar-valued unknown $\theta$ is the potential temperature, and $u=D^{-1} \nabla^{\perp} \theta$
corresponds to the velocity field of a fluid which is incompressible. One can write
$u=(-R_2\theta, R_1\theta)$ where $R_j$ is the $j^{\operatorname{th} }$ Riesz transform in
2D.  The dissipative quasi-geostrophic equation \eqref{1} can be derived from general 
quasi-geostrophic equations in the special case of constant potential vorticity and buoyancy 
frequency \cite{Ped87}. It models the evolution of the potential temperature $\theta$ of a geostrophic fluid with velocity $u$ on the
 boundary of a rapidly rotating half space. As such it is often termed 
 surface quasi-geostrophic equations in the literature.  If $\theta$ is a smooth solution to
 \eqref{1}, then it obeys the $L^p$-maximum principle, namely
 \begin{align} \label{1.2}
 \| \theta (t, \cdot) \|_{L^p(\mathbb R^2)} \le \| \theta_0 \|_{L^p(\mathbb R^2)}, 
 \qquad t\ge 0, \; \forall\, 1\le p \le \infty.
 \end{align}
 Similar results hold when the domain $\mathbb R^2$ is replaced by the periodic torus 
 $\mathbb T^2$.  Moreover, if $\theta_0$ is smooth and  in $ \dot H^{-\frac 12}(\mathbb R^2)$, then one can show that
 \begin{align} \label{1.3}
 \| \theta(t,\cdot) \|_{\dot H^{-\frac 12}(\mathbb R^2)} \le \| \theta_0 \|_{\dot H^{-\frac 12}(\mathbb R^2)},
 \quad t>0.
 \end{align}
 More precisely,  for the inviscid case $\nu=0$ one has conservation and for the dissipative
 case $\nu>0$ one has dissipation of  the $\dot{H}^{-\frac12}$-Hamiltonian. Indeed for $\nu=0$ by using the identity (below $P_{<J}$ is a smooth frequency projection to $\{|\xi|\le \mbox{constant}\cdot 2^J\}$)
$$\frac 12 \frac{d}{dt}  \| D^{-\frac 12} P_{<J} \theta \|_2^2
=-\int P_{<J} (\theta \mathcal R^{\perp} \theta) \cdot P_{<J} \mathcal R \theta dx,$$one can prove the conservation of
$\|D^{-\frac 12}
\theta\|_2^2$ under the assumption $\theta\in L_{t,x}^3$. 
 The two fundamental conservation laws \eqref{1.2} and \eqref{1.3} play important roles in the wellposedness theory for  both weak and strong solutions. In \cite{Res95}  Resnick  proved the global existence of a weak solution  for  {$0< \gamma\le 2$} in {$L_t^{\infty}L_x^2$} for any initial data 
 $\theta_0\in L_x^2$.   In \cite{Mar08} Marchand proved the existence of a global  weak solution in {$L_t^{\infty} H^{-\frac 12}_x$}  for  {$\theta_0\in \dot{H}^{-\frac 12}_x(\mathbb R^2)$} or {$L_t^{\infty} L^p_x$} for  {$\theta_0 \in L^p_x(\mathbb R^2)$}, 
 {$p\ge \frac 43$}, when  {$\nu>0$} and {$0<\gamma\le 2$}. It should be pointed out that in Marchand's result, the non-dissipative  case {$\nu=0$} requires {$p> 4/3$}
 since the embedding {$L^{\frac 43} \hookrightarrow \dot H^{-\frac 12}$} is not compact.
 On the other hand for the diffusive case one has extra {$L_t^2 \dot H^{\frac {\gamma}2-\frac 12}$} conservation
 by construction.  In recent \cite{CLH21}, non-uniqueness of stationary weak solutions were proved
 for  $\nu\ge 0$ and $\gamma<\frac 32$.  In somewhat positive direction,  uniqueness
of surface quasi-geostrophic patches for the non-dissipative case $\nu=0$ with moving boundary satisfying the arc-chord condition was obtained in  \cite{CCG18}. 

The purpose of this work is to establish optimal Gevrey regularity in the whole supercritical regime 
$0<\gamma< 1$. We begin by explaining the meaning of super-criticality.  For $\nu>0$, the
equation \eqref{1} admits a certain scaling invariance, namely: if $\theta $ is a solution, then
for $\lambda>0$ 
\begin{align}
\theta_{\lambda}(t,x) = \lambda^{1-\gamma} \theta (\lambda^{\gamma} t, \lambda x)
\end{align}
 is also a solution. As such the critical space for \eqref{1} is $\dot H^{2-\gamma}(\mathbb R^2)$ for $0\le \gamma\le 2$.  In terms of the $L^{\infty}$ conservation law, \eqref{1} is $L^{\infty}$-subcritical for $\gamma>1$, $L^{\infty}$-critical for $\gamma=1$ and $L^{\infty}$-supercritical
 for $0<\gamma<1$. Whilst the wellposedness theory for \eqref{1} is relatively complete for
 the subcritical and critical regime $1\le \gamma\le 2$ (cf. \cite{Abi08, D10, Ju04, CV10, KNV07} and the references therein), there are very few results in the
 supercritical regime $0<\gamma<1$ (\cite{Ju06, Ma14, Wu05c, Hmi07}).  In this connection
 we mention three representative works: 1)  The  work of H. Miura \cite{Mi06} which establishes  for the first time the large data local wellposedness in the critical space $H^{2-\gamma}$; 2) The  work of H. Dong \cite{D10} which via a new set of commutator estimates establishes optimal polynomial in time smoothing
estimates for critical and supercritical quasi-geostrophic equations;
3) The  work of Biswas, Martinez and Silva \cite{BMS15} which establishes short-time 
Gevrey regularity with an exponent \textbf{strictly less than $\gamma$}, namely:
\begin{align}
\sup_{0<t<T} \| e^{\lambda t^{\frac {\alpha}{\gamma} } D^{\alpha} } \theta(t,\cdot)
\|_{\dot B^{1+\frac 2p-\gamma}_{p,q}} \lesssim \| \theta_0\|_{\dot B^{1+\frac 2p-\gamma}_{p,q}},
\end{align}
where $2\le p<\infty$, $1\le q<\infty$ and $\alpha<\gamma$.

Inspired by these preceding works, we develop in this paper an optimal local regularity theory for the super-critical quasi-geostrophic equation. Set $\nu=1$ in \eqref{1}. If we completely drop the nonlinear term and keep
only the linear dissipation term, then the linear solution is given by 
\begin{align}
\theta_{\mathrm{linear}}(t, x) = (e^{ -tD^{\gamma} } \theta_0)(t,x).
\end{align}
Formally speaking, one has the identity $e^{tD^{\gamma}}( \theta_{\mathrm{linear}}(t,\cdot) )
=\theta_0$ for any $t>0$.  This shows that the best smoothing estimate one can hope for is
\begin{align}
\| e^{tD^{\gamma} } (\theta_{\mathrm{linear} } (t,\cdot ) ) \|_X \lesssim \| \theta_0 \|_X,
\end{align}
where $X$ is a working Banach space.  The purpose of this work,  rough speaking, 
is to show that for the nonlinear local solution to \eqref{1} (say taking $\nu=1$ for simplicity
of notation), we have 
\begin{align}
\| e^{(1-\epsilon_0) t D^{\gamma} } ( \theta(t,\cdot ) ) \|_X \lesssim_{\epsilon_0} \| \theta_0 \|_X,
\end{align}
where $\epsilon_0>0$ can be taken any small number, and $X$ can be a Sobolev or Besov space.
In this sense this is the  \emph{best possible} regularity estimate for this and similar problems.

We now state in more detail the main results.  To elucidate the main idea we first showcase
the result on the prototypical $L^2$-type critical $H^{2-\gamma}$ space.  The following offers
a substantial improvement of Miura \cite{Mi06}  and Dong \cite{D10}.
To keep the paper self-contained, we give a bare-hand harmonic-analysis-free proof. The framework
we develop here can probably be applied to many other problems.
\begin{thm} \label{thm1}
Let $\nu=1$, $0<\gamma< 1$ and $\theta_0 \in H^{2-\gamma} $.  For any $0<\epsilon_0<1$, there exists $T=T(\gamma,\theta_0, \epsilon_0)>0$ and 
a unique solution $\theta \in C_t^0 H^{2-\gamma} \cap C_t^1 H^{1-\gamma} \cap L_t^2 H^{2-\frac{\gamma}2} ([0,T]\times \mathbb R^2)$ to \eqref{1}
such that $f(t,\cdot) =e^{\epsilon_0 tD^{\gamma} }\theta(t,\cdot) 
\in C_t^0 H^{2-\gamma} \cap L_t^2 H^{2-\frac{\gamma}2} ([0,T]\times \mathbb R^2)$ and
\begin{align*}
\sup_{0<t<T} \|f(t,\cdot) \|_{H^{2-\gamma} }^2 + \int_0^{T} \| f(t,\cdot) \|_{H^{2-\frac{\gamma}2}}^2 dt
\le C \| \theta_0\|_{H^{2-\gamma} }^2,
\end{align*}
where $C>0$ is a constant depending  on $(\gamma, \epsilon_0)$. 
\end{thm}

Our next result is devoted to the Besov case.  In particular, we resolve the problem left open in \cite{BMS15}, namely one
can push to the \textbf{optimal threshold} $\alpha=\gamma$. Moreover we cover the whole regime
$1\le p<\infty$.

\begin{thm} \label{thm2}
Let $\nu=1$, $0<\gamma<1$, $1\le p<\infty$ and $1\le q<\infty$. Assume the initial data
$\theta_0 \in B^{1+\frac 2p-\gamma}_{p,q}(\mathbb R^2)$. There exists $T=T(\gamma,\theta_0,p,q)>0$
and a unique solution $\theta \in C_t^0([0,T], B^{1+\frac 2p-\gamma}_{p,q})$ to \eqref{1}
such that $f(t,\cdot) =e^{\frac 12 t D^{\gamma}} \theta(t,\cdot) \in C_t^0([0,T], B^{1 +\frac 2p-\gamma}_{p,q})$
and
\begin{align*}
\sup_{0<t<T} \| e^{\frac 12 t D^{\gamma} } \theta(t,\cdot) \|_{B^{1+\frac 2p-\gamma}_{p,q} } 
\le C \|\theta_0\|_{B^{1+\frac 2p-\gamma}_{p,q}},
\end{align*}
where $C>0$ is a constant depending on $(\gamma, p,q)$.  
\end{thm}

The techniques introduced in this paper may apply to many other similar models  such as Burgers equations, generalized
SQG models, and Chemotaxi equations (cf. recent very interesting
works \cite{JKM21, Iwa20, CCCGW12, Z18}). 
Also there are some promising evidences that a set of nontrivial multiplier estimates can be generalized
from our work.  All these will be explored elsewhere.
The rest of this paper is organized as follows. In Section 2 we collect some preliminary materials
along with the needed proofs.  In Section 3 we give the nonlinear estimates for the $H^{2-\gamma}$ case.  In Section 4 we give the proof
of Theorem \ref{thm1}. In Section 6 we give the proof of Theorem \ref{thm2}.
\section{Notation and preliminaries}

In this section we introduce some basic notation used in this paper and collect several useful lemmas.

We define the sign function $\operatorname{sgn}(x)$ on $\mathbb R$ as:
\begin{align*}
\operatorname{sgn}(x)= \begin{cases}
1, \quad x>0, \\
-1, \quad x<0, \\
0, \quad x=0.
\end{cases}
\end{align*}

For any two quantities $X$ and $Y$, we denote $X \lesssim Y$ if
$X \le C Y$ for some constant $C>0$. The dependence of the constant $C$ on
other parameters or constants are usually clear from the context and
we will often suppress  this dependence. We denote $X \lesssim_{Z_1,\cdots, Z_N} Y$ if the implied
constant depends on the quantities $Z_1,\cdots, Z_N$.  We denote $X\sim Y$ if $X\lesssim Y$ and $Y \lesssim X$. 

For any quantity $X$, we will denote by $X+$ the quantity $X+\epsilon$ for some sufficiently small
$\epsilon>0$. The smallness of such $\epsilon$ is usually clear from the context. The notation $X-$ is similarly defined. 
This notation is very convenient for various exponents in interpolation inequalities.
For example instead of writing
\begin{align*}
\| f g \|_{L^1(\mathbb R)} \le \| f \|_{L^{\frac {2}{1+\epsilon} } (\mathbb R)}
 \| g \|_{L^{\frac 2 {1-\epsilon} } (\mathbb R)},
 \end{align*}
 we shall write
 \begin{align} \label{eqx+x-}
 \|f g \|_{L^1(\mathbb R)} \le \| f \|_{L^{2-}(\mathbb R)} \| g \|_{L^{2+}(\mathbb R)}.
 \end{align}

For any two quantities $X$ and $Y$, we shall denote $X\ll Y$ if
$X \le c Y$ for some sufficiently small constant $c$. The smallness of the constant $c$ (and its dependence on other parameters) is
usually clear from the context. The notation $X\gg Y$ is similarly defined. Note that
our use of $\ll$ and $\gg$ here is \emph{different} from the usual Vinogradov notation
in number theory or asymptotic analysis. 

We shall adopt the following notation for Fourier transform on $\mathbb R^n$:
\begin{align*}
&(\mathcal F f )(\xi) = \hat f(\xi) = \int_{\mathbb R^n} f(x) e^{-i x \cdot \xi} dx, \\
& (\mathcal F^{-1} g)(x) = \frac 1 {(2\pi)^n} \int_{\mathbb R^n} g(\xi) e^{i x\cdot \xi} d\xi.
\end{align*}
Similar notation will be adopted for the Fourier transform of tempered distributions. 
For real-valued Schwartz functions $f:\; \mathbb R^n \to \mathbb R$, $g:\, \mathbb R^n \to \mathbb R$,
the usual Plancherel takes the form (note that $\overline{\hat g(\xi)}=\hat g(-\xi)$)
\begin{align*}
\int_{\mathbb R^n} f(x) g(x) dx = \frac 1 {(2\pi)^n}
\int_{\mathbb R^n} \hat f(\xi) \hat g(-\xi) d\xi.
\end{align*}

We shall denote for $s>0$ the fractional Laplacian $D^s = (-\Delta)^{s/2} = |\nabla|^s$ as the operator
corresponding to the symbol $|\xi|^s$. For any $0\le r \in \mathbb R$, the Sobolev norm $\|f\|_{\dot H^r}$ is 
defined as
\begin{align*}
\| f \|_{\dot H^r} = \| D^r f \|_2=\| (-\Delta)^{r/2} f \|_2.
\end{align*}

We will need to use the Littlewood--Paley (LP) frequency projection
operators. To fix the notation, let $\phi_0 \in
C_c^\infty(\mathbb{R}^n )$ and satisfy
\begin{equation}\nonumber
0 \leq \phi_0 \leq 1,\quad \phi_0(\xi) = 1\ {\text{ for}}\ |\xi| \leq
1,\quad \phi_0(\xi) = 0\ {\text{ for}}\ |\xi| \geq 7/6.
\end{equation}
Let $\phi(\xi):= \phi_0(\xi) - \phi_0(2\xi)$ which is supported in $\frac 12 \le |\xi| \le \frac 76$.
For any $f \in \mathcal S^{\prime}(\mathbb R^n)$, $j \in \mathbb Z$, define
\begin{align*}
 &\widehat{P_{\le j} f} (\xi) = \phi_0(2^{-j} \xi) \hat f(\xi), \quad P_{>j} f=f-P_{\le j} f,\\
 &\widehat{P_j f} (\xi) = \phi(2^{-j} \xi) \hat f(\xi), \qquad \xi \in \mathbb R^n.
\end{align*}
Sometimes for simplicity we write $f_j = P_j f$, $f_{\le j} = P_{\le j} f$, and $f_{[a,b]}=\sum_{a\le j \le b} f_j$.
Note that by using the support property of $\phi$, we have $P_j P_{j^{\prime}} =0$ whenever $|j-j^{\prime}|>1$.
For $f \in \mathcal S^{\prime}$ with $\lim_{j\to -\infty} P_{\le j} f =0$, one has the identity 
\begin{align*}
f= \sum_{j \in \mathbb Z} f_j , \quad ( \text{in $\mathcal S^{\prime}$} )
\end{align*}
and for general tempered distributions the convergence (for low frequencies) should be taken
as modulo polynomials.

The Bony paraproduct for a pair of functions
$f,g\in \mathcal S(\mathbb R^n) $ take the form
\begin{align*}
f g = \sum_{i \in \mathbb Z} f_i  g_{[i-1,i+1]} + \sum_{i \in \mathbb Z} f_i g_{\le i-2} + \sum_{i \in \mathbb Z}
g_i f_{\le i-2}.
\end{align*}

For $s\in \mathbb R$, $1\le p,q \le \infty$,  the Besov norm $\| \cdot \|_{B^s_{p,q}}$ is given by
\begin{align*}
\| f \|_{B^s_{p,q}} =
\begin{cases}
 \| P_{\le 0} f \|_p + ( \sum_{k=1}^{\infty} 2^{sqk} \| P_k f \|_p^q )^{1/q}, \quad \text{$q<\infty$}; \\
 \| P_{\le 0} f \|_p + \sup_{k\ge 1} 2^{sk} \|P_k f \|_p, \quad \text{$q=\infty$}.
 \end{cases}
\end{align*}
The Besov space $B^s_{p,q}$ is then simply
\begin{align*}
B^s_{p,q} = \biggl\{ f: \; f\in \mathcal S^{\prime}, \, \| f\|_{B^s_{p,q}} <\infty \biggr\}.
\end{align*}
Note that Schwartz functions are dense in $B^s_{p,q}$ when $1\le p,q<\infty$. 

In the following lemma we give refined heat flow estimate and frequency localized Bernstein
inequalities for the fractional Laplacian $|\nabla|^{\gamma}$, $0<\gamma<2$. Note that for
$\gamma>2$ and $p\ne 2$ there are counterexamples to the frequency
 Bernstein inequalities (cf.  Li and Sire \cite{LY23}). 

\begin{lem}[Refined heat flow estimate and Bernstein inequality, case $0<\gamma<2$] \label{lem_heat}
Let the dimension $n\ge 1$. Let $0<\gamma<2$ and $1\le q\le \infty$. Then for any
$f\in L^q(\mathbb R^n)$,  and any $j \in \mathbb Z$, we have
\begin{align} \label{lem_heat_e0}
\| e^{-t |\nabla|^{\gamma}} P_j f \|_q \le e^{-c_1 t 2^{j\gamma} } \| P_j f \|_q, \quad \forall\, t\ge 0,
\end{align}
where $c_1>0$ is a constant depending only on ($\gamma$, $n$). 
For $0<\gamma<2$, $1\le q<\infty$, we have
\begin{align}
&\int_{\mathbb R^n} (|\nabla|^{\gamma} P_j f) |P_j f|^{q-2} P_j f  dx \ge c_2 2^{j\gamma} \| P_j f \|_q^q, \quad \text{if
$1<q<\infty$}; \label{qnot1}\\
&\int_{\mathbb R^n} (|\nabla|^{\gamma} P_j f) \operatorname{sgn} (P_j f ) dx \ge c_2 2^{j\gamma} \|P_j f \|_1, \label{q=1}
\quad \text{if $q=1$}, 
\end{align}
where $c_2>0$ depends only on ($\gamma$, $n$). 

The $q=\infty$ formulation of \eqref{qnot1} is as follows. Let $0<\gamma<2$.  For any 
$f \in L^{\infty}(\mathbb R^n)$, if $j \in \mathbb Z$ and 
$|(P_j f)(x_0)| = \| P_j f \|_{\infty}$,  then we have
\begin{align} \label{lem_heat_e4}
\operatorname{sgn}( P_j f(x_0) ) \cdot (|\nabla|^{\gamma} P_j f)(x_0)
\ge c_3 2^{j\gamma} \| P_j f \|_{\infty},
\end{align}
where $c_3>0$ depends only on ($\gamma$, $n$). 

\end{lem}
\begin{rem*}
For $1<q<\infty$ the first two inequalities also hold for $\gamma=2$, one can see Proposition
\ref{prop2.5} and Proposition \ref{prop2.7} below. On the other hand, the inequality \eqref{q=1}
does not hold for $\gamma=2$. One can construct a counterexample in dimension $n=1$ as
follows. Take $g(x) = \frac 14 (3 \sin x -\sin 3x) = (\sin x)^3$ which only has zeros of third order.
Take $h(x) $ with $\hat h$ compactly supported in $|\xi| \ll 1$ and $h(x)>0$ for all $x$. Set
\begin{align*}
f(x) = g(x) h(x)
\end{align*}
which obviously has frequency localized to $|\xi| \sim 1$ and have same zeros as $g(x)$. Easy to check that $\|f\|_1 \sim 1$ but
\begin{align*}
\int_{\mathbb R} f^{\prime\prime}(x) \operatorname{sgn} (f(x) ) dx  = 0.
\end{align*}
\end{rem*}

\begin{rem} \label{rem2.2_1}
For $\gamma>0$ sufficiently small, one can give a direct proof for $1\le q <\infty$ as follows. WLOG consider $g=P_1 f$
with $\| g \|_q =1$, and let
\begin{align*}
I(\gamma) = \int_{\mathbb R^n} (|\nabla|^{\gamma} g ) |g|^{q-2} g dx.
\end{align*}
One can then obtain
\begin{align*}
I(\gamma)  -I(0) = \int_{\mathbb R^n} (\int_0^{\gamma}  T_{s } g  ds)  |g|^{q-2} g dx,
\quad \widehat{T_{s} g}(\xi) =s (|\xi|^{s} \log |\xi| ) \hat g (\xi).
\end{align*}
Since $g$ has Fourier support localized in $\{ |\xi| \sim 1 \}$, one can obtain uniformly
in $0<s\le 1$,
\begin{align*}
\|T_s g\|_q \lesssim_{n}  \|g \|_q =1.
\end{align*}
Note that $I(0)=1$.
Thus for $\gamma<\gamma_0(n)$ sufficiently small one must have $\frac 12 \le I(\gamma)\le \frac 32$. 
\end{rem}
\begin{rem*}
The inequality \eqref{lem_heat_e4} was obtained by Wang-Zhang \cite{WZCMP11} by an elegant
contradiction argument under
the assumption that $f\in C_0(\mathbb R^n)$ (i.e. vanishing at infinity) and $f$ is frequency localized
to a dyadic annulus.  Here we only assume $f \in L^{\infty}$ and is frequency localized. This will naturally include periodic functions and similar ones as special cases. Moreover we provide
two different proofs. The second proof is self-contained and seems quite short. 
\end{rem*}
\begin{proof}[Proof of Lemma \ref{lem_heat}]
For the first inequality and \eqref{qnot1}, see \cite{L13} for a proof using an idea of perturbation of
the L\'evy semigroup. Since the constant $c_2>0$ depends only on $(\gamma,n)$, the inequality \eqref{q=1} can
be obtained from \eqref{qnot1} by taking the limit $q\to 1$. (Note
that since $f_j = P_j f \in L^1$ and has compact Fourier support, $f_j$
 can be extended to be an entire function on $\mathbb C^n$ and
its zeros must be isolated.)

Finally for \eqref{lem_heat_e4} we give two proofs. With no loss we can assume $j=1$ and write
$f=P_1 f$. By using translation  we may also assume $x_0=0$. With no loss we assume $\|f\|_{\infty}=
f(x_0)=1$. 

The first proof is to use \eqref{lem_heat_e0} which yields
\begin{align*}
(e^{-t |\nabla|^{\gamma}} f )(0) \le e^{-c t},
\end{align*}
where $c>0$ depends only on ($\gamma$, $n$). Then since $f=P_1 f$ is smooth and
\begin{align*}
f - e^{-t |\nabla|^{\gamma} } f = \int_0^t e^{-s |\nabla|^{\gamma} }  |\nabla|^{\gamma} f
ds = t  |\nabla|^{\gamma} f - \int_0^t ( \int_0^s  e^{-\tau |\nabla|^{\gamma} } 
{ |\nabla|^{2\gamma} f } d \tau) ds.
\end{align*}
One can then divide by $t\to 0$ and obtain
\begin{align*}
 (|\nabla|^{\gamma} f) (0) \ge c.
 \end{align*}

The second proof is more direct. We note that $\int \psi(y) dy =0$ where $\psi$ corresponds to
the projection operator $P_1$.  Since $1= (P_1 f)(0)$, we obtain
\begin{align*}
1 = \int \psi (y) ( f(y) -1) dy \le \sup_{y \ne 0}( |\psi(y) | \cdot |y|^{n+\gamma} )
\cdot \int_{\mathbb R^n} \frac {1- f(y)} { |y|^{n+\gamma} } dy \lesssim_{\gamma,n}
\int_{\mathbb R^n} \frac {1-f(y) } { |y|^{n+\gamma} } dy.
\end{align*}
Thus
\begin{align*}
( |\nabla|^{\gamma} f)(0) \gtrsim_{\gamma,n } 1.
\end{align*}
\end{proof}

In what follows we will give a different proof of \eqref{qnot1} (and some stronger versions,
see Proposition \ref{prop2.5}
and Proposition \eqref{prop2.7}) and some equivalent characterization.
For the sake of understanding (and keeping track of constants) we provide some details.

\begin{lem} \label{lem2.2_equiv}
Let $0<s<1$. Then for any $g \in L^2(\mathbb R^n)$ with ${\hat g}$
being compactly supported, we have
\begin{align*}
\frac 12 C_{2s,n} \cdot \int_{\mathbb R^n \times \mathbb R^n} \frac { |g(x)-g(y)|^2} {|x-y|^{n+2s} } dx dy
= \| \,|\nabla|^s g \|^2_2,
\end{align*}
where $C_{2s,n}$ is a constant corresponding to the fractional Laplacian $|\nabla|^{2s}$ 
having the asymptotics $C_{2s,n} \sim_{n} s(1-s)$ for $0<s<1$. 
As a result, if $g \in L^2(\mathbb R^n)$ and $\|\, |\nabla|^s g \|_2<\infty$, then
\begin{align*}
s(1-s)\cdot \int_{\mathbb R^n \times \mathbb R^n} \frac { |g(x)-g(y)|^2} {|x-y|^{n+2s} } dx dy
\sim_n  \| \,|\nabla|^s g \|^2_2.
\end{align*}
Similarly if $g \in L^2(\mathbb R^n)$ and  $\int \frac{|g(x)-g(y)|^2} {|x-y|^{n+2s} } dx dy<\infty$,
then
\begin{align*}
\|\, |\nabla|^s g \|_2^2 \sim_n s(1-s)\cdot \int_{\mathbb R^n \times \mathbb R^n} \frac { |g(x)-g(y)|^2} {|x-y|^{n+2s} } dx dy.
\end{align*}

\end{lem}
\begin{proof}
Note that
\begin{align*}
\| |\nabla|^s g \|_2^2  & = \int_{\mathbb R^n }  (|\nabla|^{2s} g)(x) g(x) dx 
= \int_{\mathbb R^n} \Bigl( \lim_{\epsilon \to 0} 
 C_{2s,n} \int_{|y-x|>\epsilon} \frac {g(x)-g(y)} {|x-y|^{n+2s} } dy \Bigr) g(x) dx, \notag 
\end{align*}
where $C_{2s,n}\sim_n s (1-s)$. Now for each $0<\epsilon<1$, it is easy to check that 
(for the case $\frac 12\le s<1$ one needs to make use  of the regularised quantity $g(x)-g(y)  + \nabla g(x) \cdot (y-x)$)
\begin{align*}
| \int_{|y-x|>\epsilon}
\frac {g(x)-g(y)} { |x-y|^{n+2s} } dy | \lesssim\; |g(x)| + |\mathcal Mg(x)| + 
|\mathcal M ( \partial g )(x)| + |\mathcal M ( \partial^2 g)(x)|,
\end{align*}
where $\mathcal Mg$ is the usual maximal function. By Lebesgue Dominated Convergence,
we then obtain
\begin{align*}
\| |\nabla|^s g \|_2^2 
& = C_{2s,n} \lim_{\epsilon \to 0}
\int_{\mathbb R^n} \int_{|y-x|>\epsilon}
\frac{ g(x)-g(y)} { |x-y|^{n+2s} } dy g(x) dx.
\end{align*}
Now note that for each $\epsilon>0$, we have
\begin{align*}
\int_{\mathbb R^n} \int_{|y-x|>\epsilon} \frac {|g(x)-g(y)|} {|x-y|^{n+2s}} dy |g(x) |dx 
\lesssim_{\epsilon,n} \int_{\mathbb R^n} (|g(x) |+ |\mathcal M g(x) |) |g(x)| dx <\infty.
\end{align*}
Therefore by using Fubini, symmetrising in $x$ and $y$ and Lebesgue Monotone
Convergence, we obtain
\begin{align*}
\| |\nabla|^s g\|_2^2 = \frac 12 C_{2s,n}
\lim_{\epsilon \to 0} \int_{ |x-y|>\epsilon} \frac {|g(x)-g(y)|^2} {|x-y|^{n+2s} } dx dy
=\frac 12 C_{2s,n} \int \frac{|g(x)-g(y)|^2} {|x-y|^{n+2s}} dx dy.
\end{align*}
Now if $g \in L^2(\mathbb R^n)$ with $\| \, |\nabla|^s g \|_2<\infty$, then by Fatou's Lemma,
we get
\begin{align*}
s(1-s) \int \frac{|g(x)-g(y)|^2} {|x-y|^{n+2s} } dx dy
\le s(1-s) \cdot \liminf_{J\to \infty} \int \frac {|P_{\le J} g (x) - P_{\le J} g (y) |^2}
{|x-y|^{n+2s} } dx dy 
\lesssim_n \, s(1-s) \| \, |\nabla|^s g \|_2^2.
\end{align*}
On the other hand, note that
\begin{align*}
| P_{\le J} g(x) - P_{\le J} g(y)|
& \le \int |g(x-z)-g(y-z)| 2^{nJ} |\phi(2^J z) |d z 
\lesssim_n \left( \int |g(x-z)-g(y-z)|^2 2^{nJ} |\phi(2^J z) | dz \right)^{\frac 12},
\end{align*}
where $\phi \in L^1$ is a smooth function used in the kernel $P_{\le J}$.  The desired
equivalence then easily follows.
\end{proof}
\begin{lem} \label{lem_ab_1}
Let $1<q<\infty$. Then for any $a$, $b\in \mathbb R$, we have
\begin{align*}
& (a-b)(|a|^{q-2} a - |b|^{q-2} b)  \sim_{q} (|a|^{\frac q2-1} a - |b|^{\frac q2-1} b)^2,\\
& (a-b) (|a|^{q-2}a - |b|^{q-2}b)
\ge \frac {4(q-1)} {q^2} ( |a|^{\frac q2-1} a
- |b|^{\frac q2-1} b)^2.
\end{align*}
\end{lem}
\begin{proof}
The first inequality is easy to check. To prove the second inequality, it suffices to show
for any $0<x<1$, 
\begin{align*}
\frac {1+x^q-x-x^{q-1} } {(1-x^{\frac q2} )^2} \ge \frac {4(q-1)}{q^2} = \frac {q^2-(q-2)^2}
{q^2}.
\end{align*}
Set $t= x^{\frac q2} \in (0,1)$. The inequality is obvious for $q=2$. If $2< q<\infty$, then we need
to show
\begin{align*}
\frac{ t^{\frac 1q} - t^{1-\frac 1q} } {1-t} < \frac {q-2} q.
\end{align*}
If $1<q<2$, then we need 
\begin{align*}
\frac{t^{1-\frac 1q} -t^{\frac 1q} }{1-t} < \frac {2-q} q.
\end{align*}
 Set $\eta= \min\{ \frac 1q, 1-\frac 1q\} \in (0,\frac 12)$. It then suffices for us to show the inequality
\begin{align*}
f(\eta)= 1-2\eta - \frac {t^{\eta} - t^{1-\eta} } {1-t} \ge 0.
\end{align*}
Note that $f(0)=f(1/2)=0$ and $f^{\prime\prime}(\eta)= -(t^{\eta}-t^{1-\eta}) (\log t)^2/(1-t) <0$.
Thus the desired inequality follows.
\end{proof}

\begin{prop} \label{prop2.5}
Let $1<q<\infty$ and $0<\gamma \le 2$. Then for any $f \in L^q(\mathbb R^n)$ and
any $j\in \mathbb Z$, we have 
\begin{align}
&\int_{\mathbb R^n} (|\nabla|^{\gamma} P_j f ) |  P_j f |^{q-2} P_j f dx 
\sim_{q} \Bigl\|  |\nabla|^{\frac {\gamma} 2} ( |P_j f |^{\frac q2-1} P_j f ) \Bigr\|_2^2, \notag\\
&\int_{\mathbb R^n} (|\nabla|^{\gamma} P_j f ) |  P_j f |^{q-2} P_j f dx 
\ge \frac{4(q-1)} {q^2} \Bigl\|  |\nabla|^{\frac {\gamma} 2} ( |P_j f |^{\frac q2-1} P_j f ) \Bigr\|_2^2 \ge \frac{4(q-1)}{q^2}
\Bigl\|  |\nabla|^{\frac {\gamma} 2} ( |P_j f |^{\frac q2} ) \Bigr\|_2^2.
  \label{prop2.4_e1}
\end{align}
Consequently if $\| P_j f \|_q =1$, then for any $0<s\le 1$,
\begin{align} \label{prop2.4_e2}
\| \, |\nabla|^s ( |P_j f |^{\frac q2-1} P_j f) \|_2 \sim_{q,n} 2^{js}.
\end{align}
Also for any $0<s\le 1$,
\begin{align} \label{prop2.4_e3}
\| \, |\nabla|^s ( |P_j f |^{\frac q2} ) \|_2 \sim_{q,n} 2^{js}.
\end{align}
\end{prop}
\begin{rem*}
In \cite{Ju_CMP05}, by using a strong nonlocal pointwise inequality (see also C\'ordoba-C\'ordoba \cite{CC_CMP04}),
 Ju Proved an inequality of the form: if $0\le \gamma \le 2$, $2\le q<\infty$,
$\theta$, $|\nabla|^{\gamma} \theta \in L^q$, then
\begin{align*}
\int (|\nabla|^{\gamma} \theta) |\theta|^{q-2} \theta dx \ge \frac 2 q \|  |\nabla|^{\frac{\gamma}2} ( |\theta|^{\frac q2} ) \|_2^2.
\end{align*}
A close inspection of our proof below shows that the inequality \eqref{prop2.4_e1}  also works with $P_j f$ replaced
by $\theta$.  Note that the present form works for any $1<q<\infty$. Furthermore
in the regime $q>2$, we have $\frac {4(q-1)} {q^2} > \frac 2q$ and hence the constant here
is slightly sharper.
\end{rem*}
\begin{rem*}
The inequality \eqref{prop2.4_e2} was already obtained by Chamorro and P. Lemari\'e-Rieusset in \cite{CL12}.  Remarkably modulo a $q$-dependent constant it is equivalent to the corresponding
inequality for the more localized quantity $\int ( |\nabla|^{\gamma} P_j f ) |P_j f|^{q-2} P_j f dx$.
The inequality \eqref{prop2.4_e3} for $q>2$ was obtained
by Chen, Miao and Zhang \cite{CMZ07} by using Danchin's inequality 
$\|  \nabla ( |P_1 f|^{q/2} ) \|_2^2 \sim_{q,n} \| P_1 f \|_q^q$ together with a fractional Chain
rule in Besov spaces. The key idea in \cite{CMZ07} is to show $\|\nabla P_{[N_0, N_1]}
( |P_1 f |^{q/2} ) \|_2 \gtrsim 1$   and in order to control the high frequency piece one needs
the assumption $q>2$ (so as to use $|\nabla|^{1+\epsilon_0}$-derivative for $\epsilon_0>0$
sufficiently small).  Our approach here is different: namely we will not use Danchin's inequality
and prove directly $\|\, |\nabla|^{s_0} ( |P_1 f|^{q/2} )\|_2 \gtrsim 1$ for some $s_0$ sufficiently
small (depending on $(q,n)$).  Together with some further interpolation argument we
 are able to settle the full range $1<q<\infty$. One should note that in terms of lower bound
 the inequality \eqref{prop2.4_e3} is stronger than \eqref{prop2.4_e2}.
\end{rem*}
\begin{proof}
With no loss we can assume $j=1$ and for simplicity write $P_1 f$ as $f$. 
Assume first $0<\gamma<2$. Then for some constant $C_{\gamma,n} \sim_n \gamma(2-\gamma)$,
we have (the rigorous justification of the computation below follows a smilar argument 
as in the proof of Lemma \ref{lem2.2_equiv})
\begin{align*}
\int (|\nabla|^{\gamma} f ) |f|^{q-2} f dx
& = C_{\gamma,n} \int (\lim_{\epsilon \to 0}
\int_{|y-x|>\epsilon} \frac {f(x) -f(y)} { |x-y|^{n+\gamma} } dy ) 
|f|^{q-2} f (x) dx \notag \\
& = \frac 12 C_{\gamma,n} \int
\frac { (f(x)-f(y)) ( |f|^{q-2} f (x) - |f|^{q-2} f(y) )} {|x-y|^{n+\gamma} } dx dy 
\notag \\
& \ge \frac{4(q-1)}{q^2} \cdot \frac 12 C_{\gamma,n} \int \frac { (|f|^{\frac q2-1} f(x) -|f|^{\frac q2-1} f(y) )^2}
{ |x-y|^{n+\gamma} } dx dy  = \, \frac{4(q-1)} {q^2}  \| \,|\nabla|^{\frac {\gamma}2}
( |f|^{\frac q2-1} f ) \|_2^2,
\end{align*}
where in the last two steps we have used Lemma  \ref{lem_ab_1} and Lemma
\ref{lem2.2_equiv} respectively.
One may then carefully take the limit $\gamma \to 2$ to get the result for $\gamma=2$ (when estimating $\| |\nabla|^{\frac {\gamma} 2} ( |f|^{\frac q2-1} f)\|_2$, one needs to 
split into $|\xi| \le 1$ and $|\xi|>1$, and use Lebesgue Dominated Convergence and Lebesgue
Monotone Convergence respectively). 
By the simple inequality $|\,  |f|^{\frac q2-1} f(x) - |f|^{\frac q2-1} f(y)|
\ge |\, |f|^{\frac q2}(x)- |f|^{\frac q2}(y) |$,  we also obtain
$\| |\nabla|^{\frac {\gamma}2 } ( |f|^{\frac q2-1} f ) \|_2 
\ge \| |\nabla|^{\frac {\gamma}2 } (|f|^{\frac q2} ) \|_2$. 

Next to show \eqref{prop2.4_e2}, we can use Remark \ref{rem2.2_1} to 
obtain $\| |\nabla|^{s} g \|_2 \sim_{q, n} 1$ for any $0<s\le s_0(n)$ and $g = 
|f|^{\frac q2-1} f$. Since $\|g\|_2=1$ and $\| \nabla g \|_2 \lesssim_{q,n}1 $,
a simple interpolation argument then yields $\| |\nabla|^s g \|_2 \sim_{q,n} 1$ uniformly
for $0<s\le 1$.

Finally to show \eqref{prop2.4_e3}, we first use the simple fact that $\| \nabla (|g|) \|_2
\le \| \nabla g \|_2$ to get
\begin{align*}
\| |\nabla| ( |f|^{\frac q2} ) \|_2 \lesssim_{q} 1.
\end{align*}
It then suffices for us to show $\| |\nabla|^s ( |f|^{\frac q2} ) \|_2 \gtrsim_{q,n} 1$
for $0<s\le s_0(q,n)$ sufficiently small.  To this end we consider the quantity 
\begin{align*}
I(s) = \int_{\mathbb R^n}  |\nabla|^s (|f|)  | f|^{q-1} dx.
\end{align*}
For $0<s<1$ this is certainly well defined since $\| |\nabla|^s (|f|) \|_q \lesssim
\| f\|_q + \| \nabla ( |f| )\|_q \lesssim 1$.  To circumvent the problem of differentiating under
the integral, one can further consider the regularized expression (later $N \to \infty$)
\begin{align*}
I_N(s) = \int_{\mathbb R^n} |\nabla|^s P_{\le N} ( |f| ) | f|^{q-1} dx.
\end{align*}
Then
\begin{align*}
I_N(s)- I_N(0)= \int_{\mathbb R^n} \int_0^s (
 T_{\tilde s} P_{\le N} (|f|) ) d\tilde s  |f|^{q-1} dx, \quad \widehat{T_{\tilde s}}(\xi) =\tilde s |\xi|^{\tilde s} \log |\xi|.
 \end{align*}
Define $\widehat{T_{\tilde s}^{(1)} }(\xi) = \tilde s |\xi|^{\tilde s} (\log |\xi|)\cdot \chi_{|\xi|<1/10}$ and $T_{\tilde s}^{(2)}
=T_{\tilde s}- T_{\tilde s}^{(1)}$.  It is not difficult to check that uniformly in $0<\tilde s\le \frac 12$, 
\begin{align*}
\sup_{\xi \ne 0}  \max_{|\alpha|\le [n/2]} ( |\xi|^{|\alpha|} |\partial_{\xi}^{\alpha}
( \widehat{ T_{\tilde s}^{(1)}} (\xi) ) |) \lesssim_{n} 1.
\end{align*}
Thus by H\"{o}rmander we get $\| T_{\tilde s}^{(1)} P_{\le N} ( |f|) \|_q \lesssim_{n,q} \| f\|_q =1$. For
$T_{\tilde s}^{(2)}$ one can use $\| \nabla f \|_q \lesssim 1$ to get an upper bound which is uniform
in $0<\tilde s\le \frac 12$. Therefore $\| T_{\tilde s} P_{\le N} (|f|) \|_q \lesssim_{q,n} 1$
for $0<\tilde s\le \frac 12$. One can then obtain for $0<s\le s_0(q,n)$ sufficiently small that
$\frac 12 \le I(s) \le \frac 32$. Finally view $I(s)$ as
\begin{align*}
I(s) = \lim_{N\to \infty} \int_{\mathbb R^n}
|\nabla|^s ( Q_{\le N} (|f|) )  ( Q_{\le N} (|f|) )^{q-1} dx, 
\end{align*}
where $\widehat{ Q_{\le N} } (\xi) = \hat q(2^{-N} \xi)$, and $\hat q \in C_c^{\infty}$ satisfies
$q(x)\ge 0$ for any $x \in \mathbb R^n$ (such $q$ can be easily constructed by
taking $q(x) = \phi(x)^2$ which corresponds to $\hat q = \hat \phi * \hat \phi$).  By using the integral representation of the operator
$|\nabla|^s$ and a symmetrization argument (similar to what was done before), we can obtain
\begin{align*}
I(s) \sim_{q,n} \| \, |\nabla|^{\frac s2} ( |f|^{\frac q2} ) \|_2^2
\end{align*}
and the desired result follows.

\end{proof}

\begin{lem} \label{lem2.6}
Let the dimension $n\ge 1$ and $0<\gamma\le 2$. 
Suppose $g \in L^2(\mathbb R^n)$ and for some $N_0>0$, $\epsilon_0>0$
\begin{align*}
\|  \hat g \|_{L^2(|\xi| \ge N_0) }^2 \ge \epsilon_0 \| \hat g \|_{L^2_{\xi}(\mathbb R^n)}^2.
\end{align*}
Then  there exists $t_0=t_0(\epsilon_0, N_0,\gamma)>0$ such that for
all $0\le t \le t_0$, we have 
\begin{align*}
\| e^{ -t |\nabla|^{\gamma} } g \|_{2} \le e^{ - \frac 12 \epsilon_0 N_0^{\gamma} t } \| g \|_2.
\end{align*}
Consequently if $\tilde g \in L^2(\mathbb R^n)$ satisfies $\|\tilde g \|_2 =C_1>0$, $\| |\nabla|^s \tilde g \|_2=C_2>0$
for some $s>0$, then for any $0<\gamma\le 2$, there exists $t_0=t_0(C_1,C_2,\gamma, n, s)>0$, $c_0=
c_0(C_1,C_2, \gamma, n, s)>0$, such that 
\begin{align*}
\| e^{-t |\nabla|^{\gamma} } \tilde g \|_2 \le e^{-c_0 t} \| \tilde g \|_2.
\end{align*}
\end{lem}

\begin{proof}
With no loss we assume $\| \hat g \|_{L_{\xi}^2 } =1$. Then
\begin{align*}
\int_{\mathbb R^n} e^{ -2t |\xi|^{\gamma} } |\hat g (\xi) |^2 d\xi 
&\le 1- \| \hat g \|_{L^2 (|\xi| \le N_0)}^2 +  e^{-2 t N_0^{\gamma} }   \| \hat g \|_{L^2(|\xi|> N_0)}^2 \notag \\
&
\le  1-\epsilon_0 +e^{-2t N_0^{\gamma} } \epsilon_0
\le 1- \epsilon_0 + (1-\frac 32 tN_0^{\gamma} ) \epsilon_0 \le e^{- \epsilon_0 N_0^{\gamma} t},
\end{align*}
where in the last two steps we used the fact $e^{-x} \sim 1- x +\frac {x^2} 2$ for $x \to 0+$. 
The inequality for $\tilde g$ follows from the observation that 
$ \| |\xi|^s \hat{\tilde g}(\xi) \|_{L^2_{\xi}(|\xi| \le N_0)} \ll 1$ for $N_0$ sufficiently small. 
\end{proof}

\begin{prop} \label{prop2.7}
Let the dimension $n\ge 1$, $0<\gamma\le 2$ and $1<q<\infty$. Then for any
$f \in L^q(\mathbb R^n)$, any $j\in \mathbb Z$ and any $t>0$, we have
\begin{align*}
\| e^{ - t|\nabla|^{\gamma} } P_j f \|_q \le e^{- c 2^{j\gamma} t} \|  P_j f \|_q,
\end{align*}
where $c>0$ is a constant depending on $(\gamma,n,q)$.

\end{prop}
\begin{proof}
With no loss we assume $j=1$ and write $P_1 f $ simply as $f$. In view of the semigroup property
of $e^{-t |\nabla|^{\gamma}}$ it suffices to prove the inequality for
$0<t \ll_{\gamma,q,n} 1$. Denote $e^{-t |\nabla|^{\gamma} } f = K*f$ and observe that
$K$ is a positive kernel with $\int K(z) dz =1$. Consider first $2\le q<\infty$. Clearly
\begin{align*}
 | \int K(x-y) f(y) dy|^{\frac q2} \le \int K(x-y) |f(y)|^{\frac q2} dy.
 \end{align*}
 By Lemma \ref{lem2.6} and Proposition \ref{prop2.5}, we then get
 \begin{align*}
 \| e^{-t |\nabla|^{\gamma} } f \|_q^q 
 \le  \| e^{-t |\nabla|^{\gamma} } ( |f|^{\frac q2} ) \|_2^2 
 \le e^{-ct} \| f \|_q^q.
 \end{align*}
For the case $1<q<2$, we observe
\begin{align*}
\|\int K(x-y) |f(y)| dy \|_{L_x^q} & \le
\| ( \int K(x-y) |f(y)|^{\frac q2} dy )^{\frac {2(q-1)} q} \cdot 
( \int K(x-y) |f(y)|^q dy)^{\frac {2-q} q} \|_{L_x^q}  \notag \\
&\le \| ( \int K(x-y) |f(y)|^{\frac q2} dy)^{\frac {2(q-1)} q} \|_{L_x^{\frac q {q-1} }}
\| (\int K(x-y) |f(y)|^q d y)^{\frac {2-q} q} \|_{L_x^{\frac q {2-q} }} \notag \\
& \le \| e^{-t |\nabla|^{\gamma} } ( |f|^{\frac q2} ) \|_2^2 \cdot \| f \|_q^{2-q}.
\end{align*}
Thus this case is also OK.
\end{proof}

For the next lemma we need to introduce some terminology. Consider a function $F:\, (0,\infty) \to \mathbb R$. 
We shall say $F$ is admissible if $F\in C^{\infty}$ and
\begin{align*}
| F^{(k)} (x) | \lesssim_{k, F} x^{-k}, \quad \forall\, k\ge 0, \; 0<x<\infty.
\end{align*}
It is easy to check that $ \tilde F(x) = x F^{\prime}(x)$ is admissible if $F$ is admissible.  A simple example of
admissible function is $F(x) = e^{-x}$ which will show up in the bilinear estimates later. 

\begin{lem} \label{lem_phase_1}
Let $0<\gamma<1$ and $\sigma(\xi,\eta)= |\xi|^{\gamma} +|\eta|^{\gamma} - |\xi+\eta|^{\gamma}$ for
$\xi$, $\eta\in \mathbb R^n$, $n\ge 1$. Then for $0<|\xi|\ll 1$, $|\eta| \sim 1$, the following hold:
\begin{enumerate}
\item $ |\partial_{\xi}^{\alpha} \partial_{\eta}^{\beta} \sigma(\xi, \eta) | \lesssim_{\alpha, \beta,\gamma,n} 
|\xi|^{\gamma-|\alpha|}$,  for any $\alpha$, $\beta$. 

\item $| \partial_{\xi}^{\alpha} ( \sigma^{-m} ) | \lesssim_{\alpha,\gamma, m,n}
|\xi|^{-m\gamma-|\alpha|}$ for any $m\ge 1$ and any $\alpha$. 

\item $ |\partial_{\xi}^{\alpha} ( F(t \sigma) ) | \lesssim_{\alpha,\gamma,n, F} |\xi|^{-|\alpha|}$ for any
$\alpha$, $t>0$, and any admissible $F$. 

\item $|\partial_{\xi}^{\alpha} \partial_{\eta}^{\beta} (F(t\sigma) ) | \lesssim_{\alpha,\gamma,\beta, n, F}
|\xi|^{-|\alpha|}$, for any $\alpha$, $\beta$, and any $t>0$, $F$ admissible.
\end{enumerate}
\end{lem}
\begin{rem} \label{rem_lem_phase_1}
The condition $|\xi| \ll 1$, $|\eta|\sim 1$ can be replaced by
$0<|\xi| \lesssim 1$, $|\eta|\sim 1$, $|\xi+\eta| \sim 1$.
\end{rem}

\begin{rem*}
This lemma also highlights the importance of the assumption $0<\gamma<1$. For $0<\gamma<1$, note that the function
$g(x) =1 +x^{\gamma} - (1+x)^{\gamma} \sim  \min\{ x^{\gamma}, \, 1 \}$. By 
the triangle inequality, this implies $\sigma(\xi,\eta) \ge |\xi|^{\gamma} + |\eta|^{\gamma} - (|\xi|+|\eta|)^{\gamma}
\gtrsim  \min \{ |\xi|^{\gamma}, |\eta|^{\gamma} \}$ which does not vanish as long as $|\xi|>0$ and $|\eta|>0$. 
 However for $\gamma=1$, the
phase $\sigma(\xi,\eta) = |\xi|+ |\eta| - |\xi+\eta|$ no longer enjoys such a lower bound since $\sigma \equiv 0$ on
the one-dimensional cone $\xi = \lambda \eta$, $\lambda \ge 0$. 

\end{rem*}

\begin{proof}
With no loss we consider dimension $n=1$. The case $n>1$ is similar except some 
minor changes in numerology. 

 (1)  Note that for $|\xi|\ll 1$, $|\eta| \sim 1$, we have
\begin{align*}
\sigma(\xi, \eta) = |\xi|^{\gamma}  - \gamma \int_0^1 |\eta+ \theta \xi|^{\gamma-2} (\eta+\theta \xi) d\theta \cdot \xi.
\end{align*}
Observe that
\begin{align*}
\partial_{\eta}  \sigma(\xi, \eta) = (\operatorname{OK}) \cdot \xi,
\end{align*}
where we use the notation $\operatorname{OK}$ to denote any term which satisfy
\begin{align*}
| \partial_{\xi}^{\alpha} \partial_{\eta}^{\beta} ( \operatorname{OK} ) | \lesssim 1, \quad\forall\, \alpha,\beta. 
\end{align*}
This notation will be used throughout this proof. 
Thus for any $\beta\ge 1$, $\alpha\ge 0$, we have 
\begin{align*}
| \partial_{\xi}^{\alpha} \partial_{\eta}^{\beta} \sigma (\xi, \eta) | \lesssim |\xi|^{1-|\alpha|}.
\end{align*}
On the other hand, if $\beta=0$ and $\alpha\ge 1$, then clearly
\begin{align*}
| \partial_{\xi}^{\alpha} \sigma(\xi,\eta) | = |\partial_{\xi}^{\alpha} ( |\xi|^{\gamma} - |\xi+\eta|^{\gamma} )|
\lesssim |\xi|^{ \gamma-|\alpha|}.
\end{align*}

(2) Observe that for $|\xi| \ll 1$ and $|\eta| \sim 1$, we always have  $\sigma(\xi,\eta) \gtrsim |\xi|^{\gamma}$. One
can then induct on $\alpha$. 

(3) One can induct on $\alpha$. The statement clearly holds for $\alpha=0$. Assume the statement holds for $\alpha\le m$ and
any admissible $F$.
Then for $\alpha=m+1$, we have 
\begin{align*}
\partial_{\xi}^{m+1} ( F(t\sigma) ) = \partial_{\xi}^m (    \tilde F( t\sigma)  \cdot \sigma^{-1} \cdot \partial_{\xi} \sigma),
\end{align*}
where $\tilde F(x) = x F^{\prime}(x)$ is again admissible. The result then follows from
the inductive assumption, Leibniz and the estimates obtained
in (1) and (2).

(4) Observe that $\partial_{\eta}(\frac 1 {\sigma} ) =  - \sigma^{-2} \partial_{\eta} \sigma = \sigma^{-2} \xi \cdot
(\operatorname{OK})$, and in general for $\beta \ge 0$,
\begin{align*}
\partial_{\eta}^{\beta} ( \frac 1 {\sigma} ) = \sum_{ 0\le m\le \beta}  \sigma^{-m-1} \xi^m \cdot (\operatorname{OK}).
\end{align*}
Note that for $\beta \ge 1$ the summand corresponding to $m=0$ is actually absent (this is allowed
in our notation since we can take the term ($\operatorname{OK}$) to be zero).
Similarly one can check for any admissible $F$ and $t>0$,
\begin{align*}
\partial_{\eta}^{\beta} ( F(t\sigma) ) = \sum_{0\le m \le \beta} F_m(t\sigma) \cdot ( \frac {\xi} {\sigma} )^m \cdot 
(\operatorname{OK}),
\end{align*}
where $F_m$ are admissible functions.  This then reduce matters to the estimate in (3). The result is obvious. 
\end{proof}

\section{Nonlinear estimates: $H^{2-\gamma}$ case for $0<\gamma<1$}

\begin{lem} \label{lem1}
Set $A=D^{\gamma}$, $s=2-\gamma$ and recall
$R^{\perp}=(-D^{-1} \partial_2, D^{-1} \partial_1)$.  For any real-valued $f, \, g \in L^2(\mathbb R^2)$ with
$\hat f$ and $\hat g$ being compactly supported, it holds that
\begin{align} 
| \int_{\mathbb R^2} D^s (e^{-t A} R^{\perp} g \cdot \nabla e^{-tA} f) D^s e^{tA} f dx | \lesssim_{\gamma} 
\| g \|_{\dot H^s} \| f \|_{\dot H^{s+\frac{\gamma}2} }^2. \label{lem1_e1} 
\end{align}
If in addition $\operatorname{supp}(\hat g) \subset B(0,N_0)$ for some $N_0>0$, then
\begin{align} 
&| \int_{\mathbb R^2} D^s (e^{-t A}R^{\perp} g  \cdot \nabla e^{-tA} f) D^s e^{tA} f dx | \lesssim_{\gamma} 
N_0^{2s+2} \|f\|_2^2 \|g \|_2+ \notag \\
& \qquad \qquad
\|f \|_{\dot H^s} \cdot \| f \|_{\dot H^{s+\frac{\gamma}2}} \cdot \| g\|_2\cdot N_0^{2-\frac{\gamma}2}
+\| f \|_{\dot H^{s+\frac{\gamma}2}}^2 \cdot t N_0^2 \| g \|_2
, \label{lem1_e2}\\
%
&| \int_{\mathbb R^2} D^s (e^{-tA} R^{\perp} f \cdot \nabla e^{-tA} g) D^s e^{tA} f dx |
\lesssim_{\gamma} \| f \|_{\dot H^s}^2 \cdot N_0^2 \| g\|_2 + N_0^{2s+2} \|f \|_2^2 \|g\|_2,  \label{lem1_e3}\\
&| \int_{\mathbb R^2} D^s (e^{-tA} R^{\perp} g \cdot \nabla e^{-tA} g) D^s e^{tA} f dx |
\lesssim_{\gamma}  N_0^{2s+2} \|g \|_2^2 \|\hat f \|_{L^2(|\xi|\le 2N_0)}. \label{lem1_e4}
\end{align}
\end{lem}
\begin{rem} \label{rem_lem1_e2e3}
Note that if $\hat f$ is localized to $|\xi|\gtrsim N_0$, then the low-frequency term 
$N_0^{2s+2} \|f \|_2^2 \|g\|_2$ can be dropped in \eqref{lem1_e2} and 
\eqref{lem1_e3}.
\end{rem}
\begin{proof}
 We first show \eqref{lem1_e1}. For simplicity of notation we shall write $R^{\perp} g$ as $g$. Note that
 in the final estimates the operator $R^{\perp}$ can be easily discarded since we are in the $L^2$ setting.
 On the Fourier side we express the LHS inside the absolute value as (up to a multiplicative
 constant)
\begin{align*}
\int |\xi|^{2s} e^{-t ( |\eta|^{\gamma} + |\xi-\eta|^{\gamma} - |\xi|^{\gamma} )}
\hat g(\eta) \cdot (\xi-\eta) \hat f (\xi-\eta) \hat f (-\xi) d\xi d\eta.
\end{align*}
Observe  that by a change of variable $\xi \to \eta -\tilde \xi$ (and dropping the tildes), we have
\begin{align*}
    & \int ( \xi-\eta) |\xi|^{2s} e^{- t (|\xi-\eta|^{\gamma} - |\xi|^{\gamma} )} 
    \hat f(\xi-\eta) \hat f(-\xi) d\xi \notag \\
    = & \frac 12 
    \int \left( (\xi-\eta) |\xi|^{2s} e^{-t (|\xi-\eta|^{\gamma} - |\xi|^{\gamma} )}
    - \xi |\xi-\eta|^{2s}
    e^{-t ( |\xi|^{\gamma} - |\xi-\eta|^{\gamma})} \right) \hat f(\xi-\eta) \hat f(-\xi) d\xi.
    \end{align*}
Denote 
\begin{align*}
&\tilde \sigma(\xi,\eta)= e^{-t |\eta|^{\gamma}}
\biggl( (\xi-\eta) |\xi|^{2s} e^{-t (|\xi-\eta|^{\gamma} - |\xi|^{\gamma} )}
    - \xi |\xi-\eta|^{2s}
    e^{-t ( |\xi|^{\gamma} - |\xi-\eta|^{\gamma})} \biggr), \\
& N(g^1,g^2,g^3)  = \int \tilde \sigma(\xi,\eta) \widehat{g^1}(\eta) \widehat{g^2}(\xi-\eta) 
\widehat{g^3}(-\xi) d\xi d\eta.
\end{align*}
We just need to bound $N(g,f,f)$. By frequency localization, we have
\begin{align*}
N(g,f,f) &= \sum_j \Bigl(N(g_{<j-9}, f_j, f) + N(g_{>j+9}, f_{j},f ) + N(g_{[j-9,j+9]}, f_j, f) \Bigr).
\end{align*}
Rewriting $\sum_j N(g_{>j+9}, f_j, f) = \sum_j N(g_j, f_{<j-9}, f)$, we obtain
\begin{align*}
N(g,f,f) &=
\sum_j \Bigl(N(g_{<j-9}, f_j, f) + N(g_{j}, f_{<j-9},f ) + N(g_{[j-9,j+9]}, f_j, f) \Bigr)
\notag \\
 &= \sum_j \Bigl(N(g_{<j-9}, f_j, f_{[j-2,j+2]}) + N(g_{j}, f_{<j-9},f_{[j-2,j+2]} ) + N(g_{[j-9,j+9]}, f_j, f_{\le j+11} ) \Bigr)
 \notag \\
&= \sum_j \Bigl(  N(g_{\ll j}, f_{\sim j}, f_{\sim j}) + N(g_{\sim j}, f_{\ll j}, 
 f_{\sim j}) + N(g_{\sim j}, f_{\sim j}, f_{\lesssim j} ) \Bigr),
\end{align*}
where $g_{\ll j}$ corresponds to $|\eta| \ll 2^j$, $g_{\sim j}$ means $|\eta| \sim 2^j$, and $g_{\lesssim j}$
means $|\eta|\lesssim 2^j$. These notations are quite handy since only the relative sizes of the frequency variables $\eta$,
$\xi$ and $\xi-\eta$ will play some role in the estimates. Note that we should have written
$g_{\ll j}$ as $g_{\{l: \, 2^l \ll 2^j\} }$ according to our convention of the notation $\ll$ but we ignore
this slight inconsistency here for the simplicity of notation.

1. Estimate of $N(g_{\ll j}, f_{\sim j}, f_{\sim j})$. Note that $|\eta| \ll 2^j$, $|\xi-\eta| \sim |\xi| \sim 2^j$.
It is not difficult to check that in this regime
\begin{align*}
|\tilde \sigma(\xi,\eta)| 
&\le | (\xi-\eta)|\xi|^{2s} - \xi |\xi-\eta|^{2s} | \cdot 
e^{-t (|\eta|^{\gamma}+|\xi-\eta|^{\gamma}-|\xi|^{\gamma})}
+|\xi| |\xi-\eta|^{2s} \cdot
|e^{-t(|\eta|^{\gamma}+|\xi-\eta|^{\gamma}-|\xi|^{\gamma})}
-e^{-t (|\xi|^{\gamma}+|\eta|^{\gamma} -|\xi-\eta|^{\gamma} )} | \notag \\
& \lesssim 
2^{2js} \cdot |\eta|+|\xi| |\xi-\eta|^{2s}
|e^{-t(|\eta|^{\gamma}+|\xi-\eta|^{\gamma}-|\xi|^{\gamma})}
-e^{-t (|\xi|^{\gamma}+|\eta|^{\gamma} -|\xi-\eta|^{\gamma} )} |.
\end{align*}

Note that the second term can be bounded as (below $c>0$ is a suitably small constant)
\begin{align}
&|\xi| |\xi-\eta|^{2s}
|e^{-t(|\eta|^{\gamma}+|\xi-\eta|^{\gamma}-|\xi|^{\gamma})}
-e^{-t (|\xi|^{\gamma}+|\eta|^{\gamma} -|\xi-\eta|^{\gamma} )} | \notag \\
\lesssim  &2^{j(1+2s)} \cdot e^{-ct |\eta|^{\gamma}} \cdot t |\xi|^{\gamma-1} \cdot |\eta| \notag \\
\lesssim & 2^{j(2s+\gamma)} \cdot t |\eta| e^{-ct |\eta|^{\gamma}}. \notag 
\end{align}

To ease the notation, we write 
\begin{align}
|\sum_j N(g_{\ll j}, f_{\sim j}, f_{\sim j} ) | \lesssim \sum_j (N^{(1)} + N^{(2)}),
\end{align}
where $N^{(1)}$ is the contribution due to the term $2^{2js}|\eta|$, and $N^{(2)}$ is the
contribution due to the term  $2^{j(2s+\gamma)} \cdot t |\eta| e^{-ct |\eta|^{\gamma}}$.

Then by taking $\tilde \epsilon=\gamma$ below (note that $0<\gamma<1$), we get
(below ``$*$" denotes the usual convolution)
\begin{align*}
 \sum_j N^{(1)} 
& \lesssim \sum_j 2^{2js} \cdot \|f_{\sim j} \|_2 
\cdot ( \| \; | \widehat{Dg_{\ll j} }| * |\widehat{f_{\sim j} }|) \|_2
) \notag \\
& \lesssim \sum_j 2^{2js} \cdot \| f_{\sim j} \|_2 \cdot (\| D g_{\ll j} \|_{\dot H^{1-\tilde \epsilon} } 
\cdot \| f_{\sim j} \|_{\dot H^{\tilde \epsilon} } ) \notag \\
& \lesssim \|g \|_{\dot H^s} \| f \|_{\dot H^{s+\frac{\gamma}2} }^2.
\end{align*}
Here in the second inequality above, we have used the simple fact that 
\begin{align*}
\| |\hat A| *  |\hat B| \|_{L_{\xi}^2}
\lesssim \| \hat A \|_{L_{\xi}^1} \| \hat B \|_{L_{\xi}^2}
\lesssim \| |\xi|^{\theta} \hat A \|_{L_{\xi}^2} \| |\xi|^{1-\theta}\hat B \|_{L_{\xi}^2},
\end{align*}
if $\operatorname{supp}(\hat A) \subset \{ |\xi| \lesssim 1 \}$, 
$\operatorname{supp}(\hat B) \subset\{ |\xi| \sim 1 \}$, and $\theta<1$. 

To bound $N^{(2)}$, we note that
\begin{align}
 t \int |\eta| e^{-ct |\eta|^{\gamma}} |\widehat{g}(\eta)| d\eta 
 & \lesssim t \int  e^{-ct|\eta|^{\gamma}} |\eta|^{\gamma} \cdot |\eta|^{3-\gamma} |\widehat g(\eta)| 
 \frac {d\eta} {|\eta|^2} \notag \\
 & \lesssim t  \left(\int  e^{-2 ct |\eta|^{\gamma} } |\eta|^{2\gamma} \frac {d \eta} {|\eta|^2} \right)^{\frac 12}
 \cdot \left( \int |\eta|^{2(3-\gamma)} |\widehat g (\eta)|^2 \frac {d\eta} {|\eta|^2} \right)^{\frac 12} \notag \\
& \lesssim \| g \|_{\dot H^{2-\gamma}}.
 \end{align}
 This easily yields
 \begin{align}
 \sum_{j} N^{(2)} \lesssim \| g \|_{\dot H^s} \| f \|_{\dot H^{s+\frac {\gamma}2}}^2.
 \end{align}
 Thus
 \begin{align}
 |\sum_j N(g_{\ll j}, f_{\sim j}, f_{\sim j} ) | \lesssim 
 \| g \|_{\dot H^s} \| f \|_{\dot H^{s+\frac {\gamma}2}}^2. \notag
 \end{align}

2. Estimate of $N(g_{\sim j}, f_{\ll j}, f_{\sim j})$. In this case $|\eta| \sim |\xi| \sim 2^j$, $|\xi-\eta| \ll 2^j$. Since $s=2-\gamma\in (1,2)$, in this regime we have
\begin{align*}
|\tilde \sigma(\xi,\eta)| \lesssim |\xi-\eta| 2^{2js} + 2^{j} |\xi-\eta|^{2s} 
\lesssim 2^{2js} |\xi-\eta|.
\end{align*}
Then
\begin{align*}
| \sum_j N(g_{\sim j}, f_{\ll j}, f_{\sim j} ) |
& \lesssim \sum_j 2^{2js} \| \,  |\widehat{g_{\sim j} } | *
 |\widehat{D f_{\ll j}} | \, \|_2 
\cdot \| f_{\sim j} \|_2 \notag \\
&\lesssim \sum_j 2^{2js} \| g_{\sim j} \|_{\dot H^{\frac{\gamma}2 }}
\cdot \| D f_{\ll j} \|_{\dot H^{1-\frac{\gamma}2}}  \|f_{\sim j}\|_2
\lesssim \| g \|_{\dot H^s} \| f \|_{\dot H^{s+\frac{\gamma}2 } }^2.
\end{align*}

3. Estimate of $N(g_{\sim j}, f_{\sim j}, f_{\lesssim  j})$. In this case $|\eta| \sim |\xi-\eta| \sim 2^j$, $|\xi| \lesssim 2^j$, and
\begin{align*}
| \sum_j N(g_{\sim j}, f_{\sim j}, f_{\lesssim j } ) | &\lesssim 
 \sum_j 2^{2js} \|  |\widehat{g_{\sim j}}| * |\widehat{f_{\sim j} } |  \|_{L^{\infty}_{\xi} } \| \widehat{D f_{\lesssim j} } \|_{L^1_{\xi}} \notag \\
 & \lesssim \sum_j 2^{2js} \| g_{\sim j} \|_2 \| f_{\sim j} \|_2
 \cdot \| |\xi|^{1-\frac{\gamma} 2} 
 \widehat{Df_{\lesssim j } } \|_{L^2_{\xi}} \cdot 2^{j\frac{\gamma}2 } \notag \\ 
& \lesssim \| g \|_{\dot H^s} \| f\|_{\dot H^{s+\frac{\gamma} 2} }^2.  
 \end{align*}

 Now we turn to \eqref{lem1_e2}. Choose $J_0\in \mathbb Z$ such that $2^{J_0-1} \le N_0 <2^{J_0}$.
  Clearly by frequency localization,
 \begin{align*}
 N(g,f,f) = N(g, f_{\lesssim J_0}, f_{\lesssim J_0} ) +\sum_{j:\, 2^j \gg N_0} N(g, f_{\sim j}, f_{\sim j} ).
 \end{align*}
 For the first term we have
 \begin{align*}
 | N(g,f_{\lesssim J_0}, f_{\lesssim J_0} ) | &\lesssim N_0^{2s+1} 
 \| |\widehat{g} | *|\widehat{f_{\lesssim J_0} }| \|_{L_{\xi}^{\infty}}  \| \widehat {f_{\lesssim J_0} } \|_{L_{\xi}^1}
 \notag  \\
 & \lesssim N_0^{2s+2} \| g \|_2 \| f \|_2^2. 
 \end{align*}
 For the second term we can use the estimate of $N(g_{\ll j}, f_{\sim j}, f_{\sim j})$ and 
 break it into $N^{(1)}$ and $N^{(2)}$.
 For $N^{(1)}$, we   take $
 \tilde \epsilon=\frac{\gamma}2$  to get
 \begin{align*}
 \sum_{j: 2^j\gg N_0} N^{(1)}
 & \lesssim \sum_{j:\, 2^j \gg N_0} 2^{2js} \cdot \| f_{\sim j} \|_2
 \cdot ( \| D g \|_{\dot H^{1-\tilde \epsilon} } \| f_{\sim j} \|_{\dot H^{\tilde \epsilon} })
 \notag \\
 & \lesssim \|f \|_{\dot H^s} \cdot \| f \|_{\dot H^{s+\frac{\gamma}2}} \cdot \| g\|_2\cdot N_0^{2-\frac{\gamma}2}.
 \end{align*}
 For $N^{(2)}$ we simply bound it as
 \begin{align}
 \sum_{j:\; 2^j \gg N_0} N^{(2)} \lesssim \| f \|_{\dot H^{s+\frac{\gamma}2}}^2 \cdot t N_0^2 \| g \|_2. \notag 
 \end{align}
 These easily imply \eqref{lem1_e2}.
 
 The estimates of \eqref{lem1_e3} and \eqref{lem1_e4} are much simpler. We omit the details.
\end{proof}

\section{Proof of theorem \ref{thm1} }
To simplify numerology we conduct the proof for the case $\epsilon_0=1/2$.  Throughout this proof
we shall denote $s=2-\gamma$. 

Step 1. A priori estimate.  Denote $A= \frac 12 D^{\gamma}$ and $f= e^{t A } \theta$. It will be clear from
Step 2 below that $f$ is smooth and well-defined, and 
the following computations can be rigorously justified.
 Then $f$ satisfies
the equation
\begin{align*}
\partial_t f = -\frac 12 D^{\gamma} f - e^{t A} ( R^{\perp} e^{-tA} f \cdot \nabla e^{-t A} f ).
\end{align*}
Take $J_0>0$ which will be made sufficiently large later. Set $N_0=2^{J_0}$. Then 
\begin{align*}
\frac 12 \frac d {dt} ( \| D^s P_{>J_0} f \|_2^2 )
+ \frac 12 \| D^{s+\frac{\gamma}2} P_{>J_0} f \|_2^2  
= & -\int D^s ( R^{\perp} e^{-t A} f \cdot \nabla e^{-tA} f ) D^s e^{t A} P_{>J_0}^2 f dx.
\end{align*}
Now for convenience of notation we denote 
\begin{align*}
N(g_1,g_2,g_3) = \int D^s ( R^{\perp} e^{-tA} g_1 \cdot \nabla e^{-tA} g_2) D^s e^{tA} g_3 dx.
\end{align*}
Denote $f_h = P_{>J_0}^2 f$ and $f_l = f-f_h$. Then clearly
\begin{align*}
N(f, f, f_h) =  N(f_h, f_h, f_h) + N(f_l, f_h,f_h) + N(f_h,f_l, f_h) + N(f_l,f_l, f_h).
\end{align*}
By Lemma \ref{lem1} and noting that $\|f_l(t)\|_2 \lesssim e^{ N_0^{\gamma} t} \| \theta_0\|_2$, we get (see Remark \ref{rem_lem1_e2e3})
\begin{align*}
|N(f,f,f_h) | &\lesssim \| f_h \|_{\dot H^s} \| f_h  \|_{\dot H^{s+\frac{\gamma}2} }^2
+ \|f_h \|_{\dot H^s} \| f_h \|_{\dot H^{s+\frac{\gamma}2} }
\cdot N_0^{2-\frac{\gamma}2} e^{N_0^{\gamma} t} \| \theta_0\|_2
+ \|f_h\|_{\dot H^s}^2 \cdot N_0^2 e^{N_0^{\gamma} t} \| \theta_0\|_2 \notag \\
&\quad  + N_0^{2s+2} e^{4N_0^{\gamma} t} \| \theta_0\|_2^3+ \|f_h\|_{\dot H^{s+\frac {\gamma}2}}^2
\cdot t N_0^2 e^{N_0^2 t} \| \theta_0\|_2 .
\end{align*}
This implies for $0<t\le \min\{N_0^{-\gamma},  N_0^{-2} \tfrac {0.01}{1+\|\theta_0\|_2} \}$, 
\begin{align*}
&\frac d {dt} ( \| D^s P_{>J_0} f \|_2^2) + \biggl(\frac 13-c_1 \| D^s P_{>J_0} f\|_{2} \biggr) \cdot \| D^{s+\frac{\gamma}2} P_{>J_0} f \|_2^2 \notag \\
\le &\; \underbrace{c_2 \cdot (N_0^2 \| \theta_0\|_2 +N_0^{4-\gamma} \| \theta_0 \|_2^2) }_{=:\beta} 
\cdot\| D^s P_{>J_0} f \|_{2}^2
+c_3 N_0^{2s+2} \| \theta_0\|_2^3,
\end{align*}
where $c_1, c_2, c_3>0$ are  constants depending on $\gamma$. 

Thus as long as $\sup_{0\le s \le t} c_1 \| D^s P_{>J_0} f(s) \|_2 <\frac 1 {10}$ and $t\le 
\min\{N_0^{-\gamma},  N_0^{-2} \tfrac {0.01}{1+\|\theta_0\|_2} \}$,  we obtain
\begin{align} \label{4.1c}
\sup_{0\le s \le t} \| D^s P_{>J_0} f(s) \|_2^2 
\le  e^{\beta t} \|D^s P_{>J_0} \theta_0 \|_2^2  +t e^{\beta t} c_3 N_0^{2s+2} \| \theta_0\|_2^3. 
\end{align}

In particular, for any prescribed small constant $\epsilon_0>0$, we can first choose $J_0$ sufficiently large such that
\begin{align}
 \| D^s P_{>J_0} \theta_0 \|_2 < \frac 1 {2} \epsilon_0.
 \end{align}
 Then by using \eqref{4.1c} and choosing $T_0= T_0(J_0, \theta_0, \epsilon_0)$ sufficiently small we can guarantee 
 \begin{align} \label{4.3bb}
 \sup_{0\le s \le T_0} \| D^s P_{>J_0} f(s) \|_2 < \epsilon_0. 
 \end{align}

Step 2. Approximation system. For $n=1,2,3,\cdots$, define $\theta^{(n)}$ as solutions to the system
\begin{align*}
\begin{cases}
\partial_t \theta^{(n)} =- P_{<n} ( R^{\perp} P_{<n} \theta^{(n)} \cdot \nabla P_{<n} \theta^{(n)} ) -
D^{\gamma}  \theta^{(n)}, \\
\theta^{(n)} \bigr|_{t=0} =P_{<n} \theta_0.
\end{cases}
\end{align*}
The solvability of the above regularized system is not an issue thanks to frequency cut-offs. It is easy
to check that $\theta^{(n)}$ has frequency supported in $|\xi| \lesssim 2^n$ and $\| \theta^{(n)} \|_2 \le \| \theta_0 \|_2$.  In particular  for any $\tilde s \ge 0$ we have 
\begin{align} 
\| D^{\tilde s } P_{\le J_0} \theta^{(n) } \|_2 \lesssim 2^{c J_0 \tilde s}  \| \theta_0 \|_2,
\end{align}
where $c>0$ is a constant.

For any integer $J_0$ to be fixed momentarily, it is not difficult to check that
\begin{align*}
  & \frac 12 \frac d {dt} (\| D^s P_{>J_0} e^{tA} \theta^{(n)} \|_2^2)
  + \frac 12 \| D^{s+\frac{\gamma}2} P_{>J_0} e^{tA} \theta^{(n)} \|_2^2 \notag \\
  =&\; -
  \int D^s ( R^{\perp}P_{<n} \theta^{(n)} \cdot \nabla P_{<n} \theta^{(n)} )
  D^s e^{tA} P_{>J_0}^2  e^{tA} P_{<n} \theta^{(n)} dx.
  \end{align*}
Now fix $J_0$ sufficiently large such that
\begin{align*}
  c_1 \| D^s P_{>J_0} \theta_0 \|_2 < 1/10.
\end{align*}
By using the nonlinear estimate derived in Step 1 (easy to check that these estimates hold for $\theta^{(n)}$ with
slight changes of the constants $c_i$ if necessary), one can
then find $T_0=T_0(\gamma,\theta_0)>0$ sufficiently small such that uniformly in $n$ tending to infinity, we have
\begin{align*}
\sup_{0\le t \le T_0} \| e^{tA} D^s \theta^{(n)} (t,\cdot) \|_2 
+ \int_0^{T_0}  \| e^{tA} D^{s+\frac{\gamma}2} \theta^{(n)} (t,\cdot) \|_2^2 dt \lesssim 1.
\end{align*}
By slightly shrinking $T_0$ further  if necessary and repeating 
the argument for $\tilde A= \frac 43 A = \frac 23 D^{\gamma}$, 
we have uniformly in $n$ tending to infinity,
\begin{align*}
\sup_{0\le t \le T_0} \| e^{\frac 43 tA} D^s \theta^{(n)} (t,\cdot) \|_2  \lesssim 1.
\end{align*}
Furthermore for any prescribed small constant $\epsilon_0>0$, by using \eqref{4.3bb}, we can choose $J_0$ and
$T_0$ such that uniformly in $n$, 
\begin{align}  \notag
\sup_{0\le t \le T_0} \| D^s P_{>J_0} e^{tA} \theta^{(n)} \|_2 <\epsilon_0.
\end{align}
Note that this implies  
\begin{align} \label{4.4bb}
\sup_{0\le t \le T_0} \| D^s P_{>J_0} \theta^{(n)} \|_2 <\epsilon_0.
\end{align}
The estimate \eqref{4.4bb} will be needed later.

Step 3. Strong contraction of $\theta^{(n)}$ in $C_t^0 L_x^2$.  Denote 
$\eta_{n+1}= \theta^{(n+1)}-\theta^{(n)}$. Then (below for simplicity of notation
we write $-R^{\perp}$ as $R$)
\begin{align}
\partial_t \eta_{n+1}
&= P_{<n+1}( RP_{<n+1} \eta_{n+1} \cdot \nabla P_{<n+1} \eta_{n+1}) - D^{\gamma}
\eta_{n+1} \notag \\
& \quad + P_{<n+1} (R P_{<n+1} \theta^{(n)}
\cdot \nabla P_{<n+1}  \theta^{(n)} ) - P_{<n} (R P_{<n} \theta^{(n)}
\cdot \nabla P_{<n}  \theta^{(n)} ) \label{eq_219jan1_e00}\\
& \quad + P_{<n+1} ( RP_{<n+1} \theta^{(n)}  \cdot \nabla P_{<n+1} \eta_{n+1} ) \notag \\
& \quad +P_{<n+1} ( R P_{<n+1} \eta_{n+1} \cdot \nabla P_{<n+1} \theta^{(n)} ).
\notag
\end{align}
By using the divergence-free property, we have
\begin{align*}
\int \eqref{eq_219jan1_e00} \cdot 
\eta_{n+1} dx & = -\int 
\biggl( P_{<n+1} (R P_{<n+1} \theta^{(n)}
 P_{<n+1}  \theta^{(n)} ) - P_{<n} (R P_{<n} \theta^{(n)}
  P_{<n}  \theta^{(n)} ) \biggr)\cdot \nabla \eta_{n+1} dx.
\end{align*}
Clearly
\begin{align*}
& \|P_{<n+1} (R P_{<n+1} \theta^{(n)}
 P_{<n+1}  \theta^{(n)} ) - P_{<n} (R P_{<n} \theta^{(n)}
  P_{<n}  \theta^{(n)} ) \|_2 \notag \\
\le &\; \| P_{<n+1} ( R(P_{<n+1}-P_{<n}) \theta^{(n)}  P_{<n+1} \theta^{(n)} ) \|_2
+ \| P_{<n+1} (  RP_{<n} \theta_n  (P_{<n+1} -P_{<n}) \theta^{(n)} ) \|_2 \notag \\
& \quad + \| (P_{<n+1}-P_{<n}) (RP_{<n} \theta^{(n)} P_{<n} \theta^{(n)} ) \|_2 \notag \\
\lesssim & \; 2^{-n(2-\gamma)} \| \theta^{(n)}\|_{H^{2-\gamma} }^2 + 
2^{- n} \| \theta^{(n)}\|_{H^{2-\gamma}}^2 \lesssim 2^{-n},
\end{align*}
where we have used the uniform Sobolev estimates in Step 2. Note
that 
$$\| \nabla \eta_{n+1} \|_2 \lesssim \| \theta^{(n)} \|_{H^{2-\gamma}}
+ \| \theta^{(n+1)} \|_{H^{2-\gamma}} \lesssim 1.$$

It follows that
\begin{align*}
\frac 12 \frac d {dt}
\| \eta_{n+1}\|_2^2 + \| D^{\frac {\gamma} 2} \eta_{n+1} \|_2^2
&\le \int (RP_{<n+1} \eta_{n+1} \cdot \nabla P_{<n+1} \theta^{(n)} ) P_{<n+1} \eta_{n+1}
dx \;+\operatorname{const}\cdot 2^{-n} \notag \\
&\le \operatorname{const}\cdot 2^{-n} + \int  (RP_{<n+1} \eta_{n+1} \cdot \nabla P_{>J_0} P_{<n+1} \theta^{(n)} ) P_{<n+1} \eta_{n+1}
dx \notag \\
&\quad + \int (RP_{<n+1} \eta_{n+1} \cdot \nabla  P_{\le J_0} P_{<n+1} \theta^{(n) }) P_{<n+1} \eta_{n+1}
dx \notag \\
&\lesssim 2^{-n}+  \| \eta_{n+1} \|_{(\frac 12 -\frac{\gamma} 4)^{-1}}^2 
\cdot \| \nabla P_{>J_0} P_{<n+1} \theta^{(n)} \|_{\frac 2 {\gamma}} \notag \\
& \quad + \| \eta_{n+1} \|_2^2 \cdot 2^{2J_0} \| \theta^{(n)} \|_2 \notag \\
&\lesssim 2^{-n}+ \| D^{\frac{\gamma}2} \eta_{n+1} \|_2^2 
\cdot \| D^s P_{>J_0} \theta^{(n)} \|_2 + 2^{2J_0} \| \eta_{n+1} \|_2^2.
\end{align*}
By using the nonlinear estimates in Step 2 and \eqref{4.4bb}, one can choose $J_0$ sufficiently large 
(and slightly shrink $T_0$  further if necessary) such that
the term $\| D^s P_{>J_0} \theta^{(n)}\|_2 $ becomes sufficiently small (to kill the
implied constant pre-factors in the above inequality). This implies 
\begin{align}
 \frac d {dt} \| \eta_{n+1} \|_2^2 \lesssim 2^{-n}  + 2^{2J_0} \| \eta_{n+1} \|_2^2.
 \end{align}
Thus for some constants $\tilde c_1>0$, $\tilde c_2>0$, we have
\begin{align}
\sup_{0\le t \le T_0} \| \eta_{n+1} \|_2^2  & \le e^{\tilde c_1 \cdot 2^{2J_0} T_0}  \| \eta_{n+1}(0) \|_2^2
 + e^{\tilde c_1 \cdot 2^{2J_0} T_0} \cdot 2^{-n} \tilde c_2,
 \end{align}
 The desired  strong contraction of $\theta^{(n)} \to \theta$ in $C_t^0 L_x^2$ follows easily.

Step 4. Higher norms. By using the estimates in previous steps, we have 
for any $0\le t\le T_0$,
\begin{align*}
\| D^s e^{\frac 43t A} \theta \|_2^2 \le \limsup_{N\to \infty} \| D^s e^{\frac 43 tA} P_{\le N} \theta \|_2^2
= \limsup_{N\to \infty} \lim_{n \to \infty} \| D^s e^{\frac 43 t  A} P_{\le N} \theta^{(n)} (t) \|_2^2 < B_1<\infty,
\end{align*}
where the constant $B_1>0$ is independent of $t$. 

 It follows easily that for any $0\le s^{\prime}<s$, 
\begin{align*}
\| D^{s^{\prime}} e^{tA} (\theta^{(n)}(t) - \theta(t) ) \|_{L_t^{\infty} L_x^2} \to 0, \quad \text{as $n\to \infty$},
\end{align*}
This implies $f(t) =e^{tA} \theta(t) \in C_t^0 H_x^{s^{\prime}}$ for any $s^{\prime}<s$. To show $f \in C_t^0 H^s_x$
it suffices to consider the continuity at $t=0$ (for $t>0$ one can use the fact $e^{\frac 13 t A } f \in L_t^{\infty} H^s_x$ which
controls frequencies $|\xi| \gg t^{-1/\gamma}$, and for the part $|\xi| \lesssim t^{-1/\gamma} $ one uses $C_t^0 L_x^2$). 
Since we are in the Hilbert space setting with
 weak continuity in time,  the strong continuity then follows from norm continuity at $t=0$ which is essentially done in Step 1.

\section{Nonlinear estimates for Besov case: $0<\gamma<1$}

For $\sigma=\sigma(\xi,\eta)$ we denote the bilinear operator
\begin{align*}
T_{\sigma}(f,g)  (x) = \int_{\mathbb R^d_{\eta}} \int_{\mathbb R^d_{\xi}}
 \sigma(\xi,\eta) \hat f(\xi) \hat g(\eta) e^{i x \cdot (\xi+\eta)} d\xi d\eta.
\end{align*}
\begin{lem} \label{lem_b0}
Suppose $\operatorname{supp} (\sigma) \subset \{ (\xi,\eta):\, |\xi|< 1, \; \frac 1 {C_1}< |\eta| <C_1\}$ 
for some constant $C_1>0$. Let $n_0 =2d+[d/2]+1 $ and $\Omega_0= \{(\xi,\eta): \, 0<|\xi|<1, \;
\frac 1 {C_1} <|\eta|<C_1 \}$.  Suppose $\sigma \in C^{n_0}_{\operatorname{loc}} (\Omega_0)$ and
for some $A_1>0$
\begin{align*}
 \sup_{ \substack{|\alpha|\le [d/2]+1\\ |\beta|\le 2d} }  \sup_{(\xi,\eta) \in \Omega_0} 
  |\xi|^{|\alpha|} |\eta|^{|\beta|} 
  |\partial_{\xi}^{\alpha} \partial_{\eta}^{\beta} \sigma(\xi,\eta)| \le A_1.
  \end{align*}
 Then for any $1<p_1<\infty$, $1\le p_2\le \infty$,  $f,g \in \mathcal
S(\mathbb R^d)$, 
\begin{align*}
\| T_{\sigma}(f,g) \|_r\lesssim_{d,C_1,A_1,p_1,p_2}  \| f \|_{p_1} \| g\|_{p_2},
\end{align*}
where $\frac 1r = \frac 1 {p_1} + \frac 1 {p_2}$. 

Similarly if  $\operatorname{supp} (\sigma) \subset \{ (\xi,\eta):\,  \frac 1 {\tilde C_1}
<|\xi|<\tilde C_1, \; \frac 1 {\tilde C_2}< |\eta| <\tilde C_2\}=\Omega_1$ 
for some constants $\tilde C_1$, $\tilde C_2>0$. Suppose 
$\sigma \in C^{4d+1}_{\operatorname{loc}} (\Omega_1)$ and
for some $\tilde A_1>0$
\begin{align*}
 \sup_{ |\alpha|+|\beta|\le 4d+ 1}  \sup_{(\xi,\eta) \in \Omega_1} 
  |\xi|^{|\alpha|} |\eta|^{|\beta|} 
  |\partial_{\xi}^{\alpha} \partial_{\eta}^{\beta} \sigma(\xi,\eta)| \le \tilde A_1.
  \end{align*}
 Then for any $1\le p_1 \le \infty$, $1\le p_2\le \infty$,  $f,g \in \mathcal
S(\mathbb R^d)$, 
\begin{align*}
\| T_{\sigma}(f,g) \|_r\lesssim_{d,\tilde C_1,\tilde C_2, \tilde A_1, p_1,p_2}  \| f \|_{p_1} \| g\|_{p_2},
\end{align*}
where $\frac 1r = \frac 1 {p_1} + \frac 1 {p_2}$.


\end{lem}
\begin{proof}
For the first case see Theorem 3.7 in \cite{BMS15}. The idea is to make a Fourier expansion in the $\eta$-variable:
\begin{align*}
\sigma(\xi, \eta) =  \sum_{ k \in \mathbb Z^d} L^{-d} \int_{[-\frac L2,\frac L2]^d} \sigma(\xi, \tilde \eta)
e^{- 2\pi i \frac{k \cdot \tilde \eta} L } d\tilde \eta  \;  e^{2\pi i \frac{k \cdot \eta}L } \chi (\eta),
\end{align*}
where $L=8C_1$ and $\chi \in C_c^{\infty}( (-\frac L4,\frac L4)^d) $ is such that $\chi(\eta) \equiv 1$ for $1/C_1<|\eta|<C_1$. 
 A rough estimate
on the number of derivatives required is $n_0=2d+[d/2]+1$. Note that $r>1/2$ and (by paying $2d$ derivatives) $2dr >d$ so that the resulting summation
in $k$ converges in $l^r$-norm.  For the second case, one can make a Fourier expansion in $(\xi,\eta)$. 
\end{proof}
\begin{rem} \label{Rem5.2}
For $t>0$, $0<\gamma<1$, $j \in \mathbb Z$, consider
\begin{align*}
\sigma_0(\xi,\eta) = e^{-t (|\xi|^{\gamma} + |\eta|^{\gamma} - |\xi +\eta|^{\gamma}  )} 
\chi_{|\xi| \sim 2^j} \chi_{|\eta| \sim 2^j} \chi_{|\xi+\eta|\ll 2^j}.
\end{align*}
By using the estimates $\| \mathcal F^{-1} ( e^{t|\xi|^{\gamma} } \chi_{|\xi| \ll 1} ) \|_1 
=\| \mathcal F^{-1} ( e^{|\xi|^{\gamma} }
\chi_{|\xi| \ll t^{\frac 1{\gamma} }} ) \|_1 \lesssim e^{ \tilde ct} $ ($\tilde c\ll 1$),
$\|\mathcal F^{-1} (e^{-t |\xi|^{\gamma} } \chi_{|\xi| \sim 1} ) \|_1 \lesssim e^{- Ct} $ ($C\sim 1$), we have
for any $1\le r,p_1,p_2\le \infty$ with $\frac 1r=\frac 1 {p_1}+\frac 1{p_2}$, 
\begin{align*}
\| T_{\sigma_0} (f,g) \|_r \lesssim_{\gamma,d} e^{-c 2^{j\gamma} t} \| f \|_{p_1} \| g\|_{p_2}
\lesssim  \|f\|_{p_1} \|g\|_{p_2},
\end{align*}
where $c>0$ is a small constant.  Denote
\begin{align*}
&\sigma_1(\xi,\eta) = e^{-t(|\xi|^{\gamma} + |\eta|^{\gamma} - |\xi+\eta|^{\gamma} )}
\chi_{|\xi|\ll 2^j} \chi_{|\eta| \sim 2^j}, \\
&\sigma_2(\xi,\eta)= e^{-t( |\xi|^{\gamma} +|\eta|^{\gamma} - |\xi+\eta|^{\gamma} )}
\chi_{|\xi|\sim 2^j} \chi_{|\eta| \ll 2^j}, \\
&\sigma_3(\xi,\eta)=  e^{-t( |\xi|^{\gamma} +|\eta|^{\gamma} - |\xi+\eta|^{\gamma} )}
\chi_{|\xi|\sim 2^j} \chi_{|\eta|\sim 2^j} \chi_{|\xi+\eta| \sim 2^j}.
\end{align*}
By using Lemma \ref{lem_b0}, Lemma \ref{lem_phase_1} and some elementary computations, it is not difficult to check that for 
any $\frac 12<r<\infty$, $1<p_1,p_2<\infty$,  with $\frac 1r=\frac 1{p_1}+\frac 1 {p_2}$,
\begin{align*}
\| T_{\sigma_l} (f,g) \|_{r} \lesssim_{\gamma,p_1,p_2,d} \| f \|_{p_1} \|g \|_{p_2}, \quad\forall\, l=1,2,3.
\end{align*}
We shall need to use these inequalities (sometimes without explicit mentioning) below.
\end{rem}
\hrule
\bigskip

Fix $t>0$, $j\in \mathbb Z$, $0<\gamma<1$, and denote
\begin{align*}
  B_j(f,g)  
=& [P_j e^{t D^{\gamma} }, e^{-t D^{\gamma}} R^{\perp} f ]\cdot \nabla e^{-t D^{\gamma} } g \notag \\
=&  P_je^{t D^{\gamma}} ( e^{-tD^{\gamma} }R^{\perp} f \cdot \nabla e^{-t D^{\gamma}} g)
- e^{-t D^{\gamma} }R^{\perp} f \cdot \nabla P_j g.
\end{align*}
For integer $J_0 \ge 10$ which will be made sufficiently large later, we decompose
\begin{align}
& \;\;B_j(f, g) \notag \\
&= B_j(f_{\le J_0+2}, g_{\le J_0+4} ) + B_j(f_{\le J_0+2}, g_{>J_0+4}) 
+B_j(f_{>J_0+2}, g_{\le J_0+2}) + B_j(f_{>J_0+2}, g_{>J_0+2} ) \notag \\
&=B_j(f_{\le J_0+2}, g_{\le J_0+4} ) +B_j(f_{[J_0+2,J_0+4]}, g_{\le J_0+2}) \notag \\
& \qquad+ B_j(f_{>J_0+4}, g_{\le J_0+2})+ B_j(f_{\le J_0+2}, g_{>J_0+4}) 
 + B_j(f_{>J_0+2}, g_{>J_0+2} ). \notag 
\end{align}

\begin{lem} \label{lem_b1a}
$B_j(f_{\le J_0+2}, g_{\le J_0+4} )=0$ and $B_j(f_{[J_0+2,J_0+4]},
g_{\le J_0+2} )=0$ for $j>J_0+6$. For $j\le J_0+6$ and $1\le p<\infty$,
\begin{align*}
\| B_j(f_{\le J_0+2},  g_{\le J_0+4} ) \|_p + \| B_j(f_{[J_0+2,J_0+4]},
g_{\le J_0+2} )\|_p \lesssim e^{c_1(1+t) } (\| P_{\le  J_0+10} f \|_p^2+ \| P_{\le  J_0+10} g \|_p^2),
\end{align*}
where $c_1>0$ depends on $(J_0,p,\gamma)$.
\end{lem}
\begin{proof}
We only deal with $B_j(f_{\le J_0+2}, g_{\le J_0+4} )$ as the estimate for
$B_j(f_{[J_0+2,J_0+4]}, g_{\le J_0+2})$ is similar and therefore omitted.
Clearly for $j\le J_0+6$ (below the notation $\infty-$, $p+$ is defined
in the same way as in \eqref{eqx+x-}),
\begin{align*}
\| e^{-t D^{\gamma}} R^{\perp} f_{\le  J_0+2} \cdot \nabla P_j  g_{\le J_0+4} \|_p
\lesssim \| R^{\perp} f_{\le J_0+2}  \|_{p+} 2^{J_0} \| P_ {\le  J_0+10} g \|_{\infty-} \lesssim  2^{J_0(1+\frac 2p)} 
(\| P_{\le J_0+10} f\|_p^2 + \| P_{\le J_0+10} g \|_p^2).
\end{align*}
Here $p+$ is needed for $p=1$ so that the Riesz transform can be discarded.
On the other hand for $j\le J_0+6$,
\begin{align*}
\| P_je^{t D^{\gamma}} 
( e^{-tD^{\gamma} }R^{\perp} f_{\le J_0+2} \cdot \nabla e^{-t D^{\gamma}} g_{\le J_0+4})  \|_p
\lesssim & e^{c_1 (1+t)  }  (\| P_{\le  J_0+10} f \|_p^2 + \| P_{\le J_0+10} g \|_p^2).
\end{align*}
\end{proof}

\begin{lem} \label{lem_b1b}
For $j\le J_0+6$ and $0<t\le 1$, 
\begin{align*}
\| B_j(f_{\le J_0+2}, g_{>J_0+4}) \|_p \lesssim c_2 \| e^{-t D^{\gamma}} f_{\le J_0+2} \|_p \| g_{[J_0+5,J_0+10]} \|_p.
\end{align*}
For $j>J_0+6$, $t>0$ and $1\le p<\infty$,
\begin{align*}
\| B_j (f_{\le  J_0+2}, g_{>J_0+4} ) \|_p \lesssim c_2 \cdot 2^{j0+} \| P_{\le J_0+2} f \|_p
\| g_{[j-2,j+2]} \|_p+ c_2 (t+t^2) \| P_{\le J_0+2} f \|_p \cdot 2^{j\gamma} \| g_{[j-2,j+2]}\|_p,
\end{align*}
where $c_2>0$ depends on $(p,J_0)$, and the notation $0+$ is defined in the paragraph
preceding \eqref{eqx+x-}.
\end{lem}
\begin{proof}
The first inequality for $j\le J_0+6$ is obvious. Consider now $j>J_0+6$. 
Observe that by frequency localization $B_j(f_{\le  J_0+2}, g_{>J_0+4} ) = B_j (
f_{\le  J_0+2}, P_{>J_0+4} g_{[j-2,j+2]} )$. 
We just need to consider $T_{\sigma}(R^{\perp} f_{\le J_0+2}, \nabla P_{>J_0+4} g_{[j-2,j+2]} )$ with
$|\xi|\ll 2^j$, $|\eta| \sim 2^j$, and
\begin{align*}
\sigma(\xi,\eta) &= [\phi(2^{-j} (\xi+\eta)) e^{-t(|\xi|^{\gamma} + |\eta|^{\gamma}-
|\xi+\eta|^{\gamma} )}
- \phi(2^{-j} \eta) e^{-t |\xi|^{\gamma}}]\chi_{|\xi|\ll 2^j} \chi_{|\eta| \sim 2^j} \notag \\
& = (\phi(2^{-j} (\xi+\eta))  - \phi(2^{-j} \eta ) )e^{-t(|\xi|^{\gamma} + |\eta|^{\gamma}-
|\xi+\eta|^{\gamma} )} \chi_{|\xi|\ll 2^j} \chi_{|\eta| \sim 2^j} \notag \\
& \qquad - \phi(2^{-j} \eta)
(e^{-t(|\xi|^{\gamma} +|\eta|^{\gamma}-|\xi+\eta|^{\gamma} )} - e^{-t|\xi|^{\gamma} }) 
\chi_{|\xi|\ll 2^j} \chi_{|\eta| \sim 2^j} \notag \\
&=: \sigma_1(\xi,\eta) + \sigma_2(\xi,\eta).
\end{align*}
By Lemma \ref{lem_b0}, it is easy to check that for some $c_2>0$ depending on $(J_0,p)$, 
\begin{align*}
\|T_{\sigma_1} (R^{\perp}f_{\le J_0+2}, \nabla P_{>J_0+4} g_{[j-2,j+2]})\|_p
\lesssim   \| \partial R^{\perp}f_{\le J_0+2} \|_{\infty-} \| g_{[j-2,j+2]} \|_{p+} \lesssim c_2 \cdot
2^{j0+} \| P_{\le J_0+2} f\|_p \| g_{[j-2,j+2]}\|_p.
\end{align*}
On the other hand for $\sigma_2$, we introduce  for $0\le \tau \le 1$
\begin{align*}
F(\tau ) = e^{-t (|\xi|^{\gamma} + \tau (|\eta|^{\gamma} - |\xi+\eta|^{\gamma} )  ) }, \quad F_1(\tau)
=|\eta +\tau \xi|^{\gamma}.
\end{align*}
Then observe
\begin{align*}
F(1) - F(0) &= F^{\prime}(0 ) + \int_0^1 F^{\prime\prime}(\tau) (1-\tau) d\tau \notag \\
& =e^{-t |\xi|^{\gamma}} t (|\xi+\eta|^{\gamma} - |\eta|^{\gamma}) +
\int_0^1 F(\tau) t^2 (|\eta|^{\gamma} -|\xi+\eta|^{\gamma})^2 (1-\tau) d\tau \notag \\
& =\gamma t e^{-t|\xi|^{\gamma}} |\eta|^{\gamma-2} (\eta\cdot \xi)
+ te^{- t|\xi|^{\gamma} }  \int_0^1 F_1^{\prime\prime}(\tau) (1-\tau) d\tau
+\int_0^1 F(\tau) t^2 (|\eta|^{\gamma} -|\xi+\eta|^{\gamma})^2 (1-\tau) d\tau.
\end{align*}

To handle the last term above we make the following observation. 
Note that  for $|\xi| \ll 2^j$, $|\eta|\sim 2^j$, one has $| |\eta|^{\gamma}-|\xi+\eta|^{\gamma} |=
\mathcal O( |\eta|^{\gamma-1} \cdot |\xi| ) \le \frac {\alpha_0} 2 |\xi|^{\gamma}$, for some constant
$0<\alpha_0<1$.  In particular one may write
\begin{align}
F(\tau) = e^{-t(1-\alpha_0) |\xi|^{\gamma} } e^{-t ( \alpha_0 |\xi|^{\gamma} +\tau( |\eta|^{\gamma}
-|\xi+\eta|^{\gamma} )}.
\end{align}
The symbol corresponding to $e^{-t( \alpha_0 |\xi|^{\gamma} +\tau( |\eta|^{\gamma}
-|\xi+\eta|^{\gamma} )}$ is clearly good for us whilst the term $e^{-t (1-\alpha_0) |\xi|^{\gamma} }$ can
be used to extract additional decay (see below).

It is then clear that
\begin{align*}
\| T_{\sigma_2} (R^{\perp} f_{\le J_0+2},  \nabla P_{>J_0+4} g_{[j-2,j+2]} )
\|_p &\lesssim  t \| e^{-t D^{\gamma} } \partial R^{\perp} f_{\le  J_0+2} \|_{\infty}
\cdot 2^{j\gamma} \| g_{[j-2,j+2]} \|_p
\notag \\
& \quad+ 2^{j(\gamma-1)} 
\cdot t \| e^{-t D^{\gamma}} \partial^2 R^{\perp} f_{\le J_0+2} \|_{\infty-} \| g_{[j-2.j+2]} \|_{p+}  \notag \\
&\quad + 2^{j(2\gamma-1)} \cdot t^2 \| \partial^2 e^{-t(1-\alpha_0) D^{\gamma} }R^{\perp} f_{\le  J_0+2} 
\|_{\infty-}  \|g_{[j-2,j+2]} \|_{p+} \notag \\
& \lesssim c_2 \cdot (t+t^2) \| P_{\le J_0+2} f \|_p \cdot 2^{j\gamma} \| g_{[j-2,j+2]} \|_p.
\end{align*}
\end{proof}

\begin{lem} \label{lem_b1c}
Let $1\le p<\infty$.
For $j\le J_0+6$,  $0<t\le 1$, we have
\begin{align*}
\| B_j(f_{>J_0+4}, g_{\le J_0+2} ) \|_p 
\lesssim c_2  \| e^{-t D^{\gamma}} f_{>J_0+4} \|_p \| g_{\le J_0+2} \|_p.
\end{align*}
For $j>J_0+6$ and any $t>0$, we have
\begin{align*}
\| B_j (f_{> J_0+4}, g_{\le J_0+2} ) \|_p \lesssim c_2 \cdot 2^{j0+} \| P_{\le J_0+2} g \|_p
\| f_{[j-2,j+2]} \|_p,
\end{align*}
In the above $c_2>0$ depends on $(\gamma, p,J_0)$. 
\end{lem}
\begin{proof}
The estimate for $j\le J_0+6$ is obvious.
Observe that for $j>J_0+6$,
\begin{align*}
B_j(f_{>J_0+4}, g_{\le J_0+2} )
= P_j e^{t D^{\gamma} } ( e^{- t D^{\gamma} } R^{\perp} P_{>J_0+4} f_{[j-2,j+2]}\cdot
\nabla e^{-t D^{\gamma} } P_{\le J_0+2} g).
\end{align*}
Thus by Lemma \ref{lem_b0},
\begin{align*}
\| B_j(f_{>J_0+4}, g_{\le J_0+2} ) \|_p \lesssim \| R^{\perp} f_{[j-2,j+2]} \|_{p+} \| \nabla P_{\le  J_0+2} g \|_{\infty-}
\lesssim c_2 2^{j0+} \| P_{\le J_0+2} g \|_p \| f_{[j-2,j+2]} \|_p.
\end{align*}

\end{proof}

\begin{lem} \label{lem_b1d}
Denote $f^h= f_{>J_0+2}$, $g^h=g_{>J_0+2}$. Then for $j\ge J_0$, $1\le p<\infty$, $0<t\le 1$,  we have
\begin{align*}
\| B_j (f^h, g^h)\|_p 
& \lesssim 2^{j\gamma} \| f^h \|_{\dot B^{1+\frac 2p-\gamma}_{p,\infty}} \| g^h_{[j-2,j+2]}\|_p
+ 2^{j\gamma} \| g^h \|_{\dot B^{1+\frac 2p-\gamma}_{p,\infty}} \| f^h_{[j-2,j+9]} \|_p 
 \notag \\
 & \qquad\qquad+2^j \sum_{k\ge j+8} 2^{k\cdot \frac 2p}
\| f^h_k \|_p \| g^h_{[k-2,k+2]} \|_p.
\end{align*}
\end{lem}

\begin{proof}
Write
\begin{align*}
B_j(f^h, g^h ) & = B_j(f^h_{<j-2}, g^h) + B_j(f_{[j-2,j+9]}^h, g^h) + B_j(f^h_{>j+9}, g^h) \notag \\
& = B_j(f^h_{<j-2}, g^h_{[j-2,j+2]}) + B_j(f^h_{[j-2,j+9]}, g^h_{<j-4})
+B_j(f^h_{[j-2,j+9]}, g^h_{[j-4,j+12]}) +\sum_{k\ge j+10} B_j(f^h_k, g^h) \notag \\
& =: (1)+ (2) +(3) +(4).
\end{align*}

\underline{Estimate of $(1)$}. 

Note that for given integer $J_1\ge 2$, the term $B_j(f^h_{[j-J_1,j-3]},g^h_{[j-2,j+2]})$ can be included in the estimate of (3).

It suffices for us to estimate
$(1\text{A})=  T_{\sigma} ( R^{\perp} f^h_{<j-J_1}, \nabla g^h_{[j-2,j+2]})$ with (here
to ensure $|\xi| \ll 2^j$ we need to take $J_1$ sufficiently large)
\begin{align*}
\sigma(\xi,\eta) = ( \phi(\frac{\xi+\eta} {2^j}) -\phi(\frac{\eta}{2^j}) )
e^{-t(|\xi|^{\gamma} + |\eta|^{\gamma} - |\xi+\eta|^{\gamma} )} 
\chi_{|\xi|\ll 2^j} \chi_{|\eta| \sim 2^j} 
- \phi(\frac{\eta}{2^j}) ( e^{-t(|\xi|^{\gamma} +|\eta|^{\gamma}-
|\xi+\eta|^{\gamma})} -e^{-t|\xi|^{\gamma} })   \chi_{|\xi| \ll 2^j} \chi_{|\eta| \sim 2^j}.
\end{align*}
By an argument similar to that in Lemma \ref{lem_b1b}, we get 
\begin{align*}
\| (1\text{A}) \|_p &\lesssim \| \partial R^{\perp} f^h_{<j-2} \|_{\infty-} \| g^h_{[j-2,j+2]} \|_{p+}
+t \| e^{-t D^{\gamma} } \partial R^{\perp} f_{<j-2}^h \|_{\infty}
\cdot 2^{j\gamma} \| g_j^h\|_p+ 2^{j(\gamma-1)} 
\cdot t \| e^{-t D^{\gamma}} \partial^2 R^{\perp} f_{<j-2}^h \|_{\infty-} \| g_j^h \|_{p+}  \notag \\
&\quad + 2^{j(2\gamma-1)}  \|  D^{2-2\gamma} R^{\perp} f_{<j-2}^h 
\|_{\infty-}  \|g_j^h \|_{p+} \notag \\
& \lesssim 2^{j\gamma} \| f^h\|_{\dot B^{1+\frac 2p-\gamma}_{p,\infty}} \| g^h_{[j-2,j+2]} \|_p.
\end{align*}

\underline{Estimate of $(2)$}.  Clearly 
\begin{align*}
(2)= P_j e^{t D^{\gamma}} ( R^{\perp} e^{-t D^{\gamma} } f^h_{[j-2,j+9]}\cdot \nabla e^{-t D^{\gamma}}
g^h_{<j-4} ).
\end{align*}
Thus
\begin{align*}
\| (2) \|_p \lesssim \| f^h_{[j-2,j+9]} \|_{p+} \| \nabla g^h_{<j-4} \|_{\infty-}  
\lesssim \| g^h \|_{\dot B^{1-\gamma+ \frac 2p}_{p,\infty} }  2^{j\gamma}\| f^h_{[j-2,j+9]} \|_p.
\end{align*}

\underline{Estimate of $(3)$}.  Clearly 
\begin{align*}
\| (3) \|_p \lesssim  2^{j\gamma} \| g^h \|_{\dot B^{1-\gamma+\frac 2p}_{p,\infty} }
\| f^h_{[j-2,j+9]}\|_p.
\end{align*}

\underline{Estimate of $(4)$}. 
We first note that 
\begin{align}
 \sum_{k\ge j+10}
 \| e^{-tD^{\gamma}} R^{\perp} f_k^h \cdot \nabla P_j g \|_p
 \lesssim 
 2^{j\gamma} \| f^h \|_{\dot B^{1+\frac 2 p-\gamma}_{p,\infty}}
 \| g^h_{[j-2,j+2]} \|_p.
 \end{align}
On the other hand by using that $R^{\perp} f $ is divergence-free, we have
\begin{align*}
&\sum_{k\ge j+10}  \left\| P_j e^{t D^{\gamma} } \nabla
\cdot \biggl( (e^{-t D^{\gamma} } R^{\perp} f^h_k   ) (e^{-t D^{\gamma} } g^h_{[k-2,k+2]} ) \biggr) \right\|_p \notag \\
 \lesssim & \sum_{k\ge j+10} 2^j \| f^h_k \|_{2p} \| g^h_{[k-2,k+2]} \|_{2p} \notag \\
 \lesssim &\; 2^j \sum_{k\ge j+8} 2^{k\cdot \frac 2p} 
\| f^h_k \|_p \cdot  \| g^h_{[k-2,k+2]} \|_p.
\end{align*}
\end{proof}

\section{Proof of Theorem \ref{thm2}}
Recall that the initial data $\theta_0 \in B_{p,q}^{1-\gamma+\frac 2p}$, $1\le p<\infty$, $0<\gamma<1$ and
$1\le q<\infty$.

\begin{lem} \label{lem5.1}
Let $\chi \in C_c^{\infty}(\mathbb R^2)$ and $\theta_0 \in B_{p,q}^s(\mathbb R^2)$ with $1\le p<\infty$,
$1\le q<\infty$, $s>0$. Let $(\lambda_n)_{n=1}^{\infty}$ be a  sequence of positive numbers such
that $ \inf_n \lambda_n >0$. 
Then
\begin{align*}
\lim_{J_0\to \infty} \sup_{n\ge 1} \| P_{>J_0} \bigl( \,\chi(\lambda_n^{-1} x) P_{\le n+2} \theta_0 \,\bigr) \|_{ B^s_{p,q}}
=0.
\end{align*}
\end{lem}
\begin{proof}
Write $f= \chi(\lambda_n^{-1} x)$, $g= P_{\le n+2} \theta_0$, then 
\begin{align*}
fg = \sum_{j \in \mathbb Z} ( f_{j} g_{<j-2} + g_j f_{<j-2} + f_j \tilde g_j),
\end{align*}
where $\tilde g_j = g_{[j-2,j+2]}$. Clearly
\begin{align*}
(2^{js}\| f_j g_{<j-2} \|_p )_{l_j^q( j \ge J_0)} \lesssim 
 \|g \|_p  \| \partial^{2[s]+2} f\|_{\infty} (2^{-j(2[s]+2-s)})_{l_j^q(j\ge J_0)} \to 0,
\end{align*}
uniformly in $n$ as $J_0 \to \infty$.
A similar estimate also shows that the diagonal piece $f_j \tilde g_j$ is OK. On the other hand
\begin{align*}
( 2^{js} \| f_{<j-2} g_j \|_p ) _{l_j^q(j\ge J_0)}
\lesssim \| f\|_{\infty} (2^{js} \| P_j \theta_0 \|_p)_{l_j^q(j\ge J_0)} \to 0,
\end{align*}
uniformly in $n$ as $J_0\to \infty$.
\end{proof}

We now complete the proof of Theorem \ref{thm2}. This will be carried out in several steps below.

Step 1. Definition of approximating solutions. Define $\theta^{(0)} \equiv 0$. For $n \ge 0$, define the iterates
$\theta^{(n+1)}$ as solutions to the following system
\begin{align*}
\begin{cases}
\partial_t \theta^{(n+1)} = -R^{\perp}\theta^{(n)} \cdot \nabla \theta^{(n+1)} - D^{\gamma} \theta^{(n+1)}, 
\quad (t,x) \in (0,\infty) \times \mathbb R^2;\\
\theta^{(n+1)}\Bigr|_{t=0} = \chi(\lambda_n^{-1} x) P_{\le n+2} \theta_0,
\end{cases}
\end{align*}
where $\chi \in C_c^{\infty}(\mathbb R^2)$ satisfies  $0\le \chi \le 1$ for all $x$, 
$\chi(x)\equiv 1$ for $|x|\le 1$, and $\chi (x)=0$ for $|x|\ge 2$. 
Here we introduce the spatial cut-off $\chi$ so that $\theta^{(n+1)}\Bigr|_{t=0} \in H^k$ for all $k\ge 0$ when we only assume
$\theta_0$ lies in $L^p$ type spaces.  The scaling parameters $\lambda_n \ge 1$ are inductively chosen such that
$\lambda_n >\max\{4\lambda_{n-1},2^n\} $ and
\begin{align*}
\| \theta_0 \|_{L^p(|x|>\frac 1 {100} \lambda_n)} <2^{-100n}.
\end{align*}
Easy to check that
\begin{align*}
\| \theta^{(n+1)}(0) - \theta^{(n)} (0) \|_p \lesssim 2^{-n(1-\gamma+\frac 2p)},
\end{align*}
and by interpolation for $0<\tilde s<1-\gamma+\frac 2p$, $\tilde s=0+$,
\begin{align} \label{initial_contra}
\| \theta^{(n+1)}(0) - \theta^{(n)}(0) \|_{B^{\tilde s}_{p,\infty}} \lesssim 2^{-n(1-\gamma+\frac 2p)+}.
\end{align}
Also by Lemma \ref{lem5.1}, we have
\begin{align*}
\| \theta^{(n+1)}(0) - \theta_0 \|_{B^{1-\gamma+\frac 2p}_{p,q} } \to 0, \quad \text{as $n\to\infty$}.
\end{align*}
These estimates will be needed for the contraction estimate later.

Clearly we have the  uniform boundedness of $L^p$ norm:
\begin{align*}
\sup_{n\ge 0} \sup_{0\le t <\infty} \| \theta^{(n+1)} (t) \|_p \le \sup_{n\ge 0} \| P_{\le n+2} \theta_0 \|_p
\lesssim \| \theta_0\|_p.
\end{align*}
This will often be used without explicit mentioning below.

Step 2. 
Denote $A=\frac 12 D^{\gamma}$, $f^{(n+1)}(t)=e^{t A } \theta^{(n+1)} (t)$. Then
\begin{align*}
\partial_t f^{(n+1)} = - A  f^{(n+1)} - e^{t A} (  R^{\perp}e^{-tA} f^{(n)} \cdot \nabla e^{-tA} f^{(n+1)} ).
\end{align*}
One can view $f^{(n+1)}$ as the unique limit of the sequence of solutions $(f^{(n+1)}_m)_{m=1}^{\infty}$ solving
the regularized system
\begin{align*}
\begin{cases}
\partial_t f^{(n+1)}_m = - A f^{(n+1)}_m - e^{tA} {P_{\le m}} ( 
 R^{\perp} e^{-t A} f^{(n)} \cdot \nabla e^{-t A } P_{\le m} f^{(n+1)}_m), \\
f^{(n+1)}_m \Bigr|_{t=0} = P_{\le m} \Bigl( \chi(\lambda_n^{-1} x) P_{\le n+2} \theta_0 \Bigr).
\end{cases}
\end{align*}
By using the estimates in Section 2 (and the inductive assumption that
$f^{(n)} \in C_t^0 H^k$ for all $k\ge 0$), we can then obtain $f^{(n+1)} \in C_t^0 ([0,T], H^k)$ for all $T>0$, $k\ge 0$. 
Write
\begin{align*}
f^{(n+1)}(t) = e^{-t A} f^{(n+1)}(0) - \int_0^t e^{-(t-s) A} e^{sA} (  R^{\perp}e^{-sA}
f^{(n)} \cdot \nabla e^{-sA} f^{(n+1)} ) ds.
\end{align*}
By using the fact that $f^{(n)}$, $f^{(n+1)} \in C_t^0 H^k$, it is not difficult to check that 
\begin{align*}
\sup_{0\le t \le T} \| \partial^k f^{(n+1)}(t) \|_p <\infty, \quad \forall\, T>0,\, k\ge 0.
\end{align*}
It follows that for any $T>0$
\begin{align*}
\sup_{0\le t\le T} \| \partial_t f^{(n+1)}\|_p\le 
\sup_{0\le t \le T} \| D^{\gamma} f^{(n+1)} \|_p + \| e^{tA} ( R^{\perp} e^{-tA} f^{(n)} \cdot
\nabla e^{-t A}  f^{(n+1)} ) \|_p <\infty.
\end{align*}
This together with interpolation implies $f^{(n+1)} \in C_t^0([0,T], W^{k,p})$ for any $T>0$, $k\ge 0$.
These estimates establish the (a priori) finiteness of the various Besov norms  and associated time continuity
needed in the following steps.

Step 3. Besov norm estimates.  Denote 
$f_j^{(n+1)} = P_j f^{(n+1)}$.  For any $\epsilon_0>0$, we show that there exists
$J_1$ sufficiently large, and $T_1>0$ sufficiently small, such that
\begin{align} \label{step3_J1T1}
\sup_{n\ge 0}   ( 2^{j(1-\gamma+\frac 2p)}\|f_j^{(n)} \|_p )_{l_j^q L_t^{\infty} (t\in [0,T_1],\, j\ge J_1) } <\epsilon_0.
\end{align}

 Clearly for each $j\in \mathbb Z$,
\begin{align*}
\partial_t f^{(n+1)}_j = -\frac 12 D^{\gamma} f^{(n+1)}_j- P_j e^{tA} ( 
R^{\perp} e^{-t A } f^{(n)} \cdot \nabla e^{-t A} f^{(n+1)}).
\end{align*}
Then by using Lemma \ref{lem_heat}, we get for some constants $\tilde C_1>0$, $\tilde C_2>0$,
\begin{align*}
\partial_t ( \| f_j^{(n+1)} \|_p)
+ \tilde C_1 2^{j\gamma} \| f_j^{(n+1)} \|_p
\le \tilde C_2 \| [P_j e^{tA}, e^{-t A} R^{\perp} f^{(n)} ] \cdot \nabla e^{-t A} f^{(n+1)}\|_p.
\end{align*}
Take an integer $J_0\ge 10$ which will be made sufficiently large later. 
By using the nonlinear estimates derived before (see Lemma \ref{lem_b1a}--\ref{lem_b1d}),
we then obtain
\begin{align*}
\|f_j^{(n+1)}(t) \|_p\le e^{-\tilde C_1 2^{j\gamma} t} \|f_j^{(n+1)}(0) \|_p
+ \int_0^t e^{-\tilde C_1 2^{j\gamma} (t-s) } N_j ds,
\end{align*}
where for some constants $\tilde C_3>0$, $\tilde C_4>0$, $\tilde C_5>0$, 
\begin{align*}
N_j&= 1_{j\le J_0+6} \cdot {\tilde C_3 } \cdot
( \| P_{\le J_0+10} f^{(n)} \|_p^2 + \| P_{\le  J_0+10} f^{(n+1) } \|_p^2
+\| \theta_0\|_p^2 ) 
+ 1_{j>J_0+6} \cdot \tilde C_4 \cdot 2^{j0+} (\| P_{\le J_0+2} f^{(n)} \|_p \| f^{(n+1)}_{[j-2,j+2]} \|_p ) \notag \\
& \quad 
+ 1_{j>J_0+6} \cdot \tilde C_4 \cdot (2^{j0+}\|P_{\le J_0+2} f^{(n+1) } \|_p \| f^{(n)}_{[j-2,j+2]} \|_p
+ s\| P_{\le J_0+2} f^{(n)} \|_p \cdot 2^{j\gamma} \| f_{[j-2,j+2]}^{(n+1)} \|_p ) \notag \\
& \quad + \tilde C_5 2^{j\gamma} (\| P_{>J_0+2} f^{(n)} \|_{\dot B^{1+\frac 2p-\gamma}_{p,\infty} }
\| P_{>J_0+2} f^{(n+1)}_{[j-2,j+2]} \|_p 
+ \| P_{>J_0+2} f^{(n+1)} \|_{\dot B^{1+\frac 2p-\gamma}_{p,\infty} }
\| P_{>J_0+2} f^{(n)}_{[j-2,j+9]} \|_p ) \notag \\
& \quad + \tilde C_5 2^j \sum_{k\ge j+8} 2^{k\frac 2p} \| P_{>J_0+2} f^{(n)}_k \|_p \| 
P_{>J_0+2} f^{(n+1) }_{[k-2,k+2]}\|_p.
\end{align*}

Denote 
\begin{align*}
\| f^{(n+1)} \|_{T,J_0} = (2^{j(1-\gamma+\frac 2p)}  \|f_j^{(n+1) } \|_p)_{l_j^q L_t^{\infty} (t \in [0,T],\, j\ge J_0)}.
\end{align*}
One should note that by the estimates derived in Step 2, the above norm of $f^{(n+1)}$ is finite. 
Then for $0<T\le 1$,
\begin{align*}
\|f^{(n+1)}\|_{T,J_0}
&\le ( 2^{j(1-\gamma+\frac 2p)}\|f_j^{(n+1)}(0) \|_p)_{l_j^q(j\ge J_0)} + C_{J_0}^{(1)} T \| \theta_0\|_p^2
\notag \\
& \qquad +C_1 \cdot 2^{-\frac 12 J_0 \gamma} \cdot e^{C_3 \cdot 2^{J_0\gamma} T} \|\theta_0\|_p
\cdot ( \|f^{(n)} \|_{T,J_0} + \| f^{(n+1)} \|_{T,J_0} ) \notag \\
&\quad +  C_{J_0}^{(2)} \cdot T \|\theta_0\|_p \| f^{(n+1)}\|_{T,J_0}
+ C_2 \| f^{(n) } \|_{T, J_0} \| f^{(n+1) } \|_{T,J_0},
\end{align*}
where $C_{J_0}^{(1)}$, $C_{J_0}^{(2)}>0$ are constants depending on $(J_0,\gamma, p,q)$, 
$C_1$, $C_2>0$ are constants depending only on $(\gamma,p,q)$, and $C_3>0$
depends only on $\gamma$. 

By Lemma \ref{lem5.1}, one can find $J_0$ sufficiently large such that
\begin{align*}
&\sup_{n\ge 0} (2^{j(1-\gamma+\frac 2p)} \|f_j^{(n+1)}(0) \|_p)_{l^q_j(j\ge J_0)} < \frac 1 {100 C_2}, \\
& C_1 \cdot 2^{-\frac 12 J_0 \gamma} \cdot 10 (1+ \|\theta_0\|_p )<\frac 1{20}.
\end{align*}
Fix such $J_0$ and then choose $T=T_0\le 1$ such that 
\begin{align*}
&C_{J_0}^{(1)} T_0 \| \theta_0 \|_p^2 <\frac 1 {100 C_2}, \quad  C_3 \cdot 2^{J_0\gamma}\cdot T_0<\frac 1{100}, \quad
C_{J_0}^{(2)} T_0 \| \theta_0\|_p <\frac 1 {20}.
\end{align*}
The inductive assumption is $\|f^{(n)}\|_{T_0,J_0} <\frac 1{4C_2}$. Then clearly 
\begin{align*}
\| f^{(n+1)} \|_{T_0, J_0} \le \frac 1 {100 C_2} + \frac 1 {100 C_2} +\frac 1{20} \| f^{(n+1) }\|_{T_0,J_0}
+ \frac 1 {20} \cdot \frac 1{4C_2} + \frac 1 {20} \| f^{(n+1)} \|_{T_0,J_0} + \frac 1 4 \| f^{(n+1)} \|_{T_0,J_0}.
\end{align*}
This easily implies $\|f^{(n+1)} \|_{T_0,J_0} < \frac 1 {4C_2}$ which completes the argument.

The statement \eqref{step3_J1T1} clearly follows by a slight modification of the above argument.

Step 4. Contraction in $B^{s_0}_{p,\infty}$ where $s_0>0$ is a sufficiently small number. 
\begin{rem*}
We chose the space $C_t^0 B^{0+}_{p,\infty}$ since it contains $L^p$ and its norm coincides with the usual Chemin-Lerner
space $\tilde L_t^{\infty} B^{0+}_{p,\infty}$ (see \eqref{norm_equiv1}).
This way  one can make full use of the smoothing effect
of the linear semigroup on each dyadic frequency block which is needed for this critical problem.
\end{rem*}

Set $\eta^{(n+1)}=
f^{(n+1)}-f^{(n)}$. Then
\begin{align*}
\partial_t \eta^{(n+1)} = - A \eta^{(n+1)}
- e^{tA} (R^{\perp} e^{-tA} \eta^{(n)} \cdot \nabla e^{-tA} f^{(n+1)} ) - e^{tA}
(R^{\perp} e^{-t A} f^{(n-1)} \cdot \nabla e^{-tA} \eta^{(n+1) } ).
\end{align*}
It is easy to check for $0 \le t \le T_0$, $J_1 \in \mathbb Z$ (below we work with $p+$ to avoid the end-point situation
$p=1$)
\begin{align*}
\partial_t \| P_{\le 2J_1} \eta^{(n+1) } \|_p &\lesssim_{J_1, T_0,p,\gamma} \| e^{-tA} R^{\perp}\eta^{(n)} \|_{p+} \| f^{(n+1) } \|_{\infty-} + 
\| R^{\perp} f^{(n-1) } \|_{\infty-}  \| e^{-tA} \eta^{(n+1)} \|_{p+} \notag \\
& \lesssim_{J_1, T_0,p,\gamma} \| \eta^{(n)} \|_p \| f^{(n+1)} \|_{\infty-} + 
\|f^{(n-1)} \|_{\infty-} \| \eta^{(n+1)} \|_p
\notag \\
&\le C_{T_0, J_1} \cdot (\| \eta^{(n) } \|_p + \| 
\eta^{(n+1)} \|_p),
\end{align*}
where $C_{T_0,J_1}$ is a constant depending only on $(\theta_0, J_1, T_0, \gamma, p, q)$. Here in the last inequality 
we used the estimates
obtained in Step 3.

On the other hand for $j\ge J_1$, denoting $\eta^{(n+1)}_j =P_j \eta^{(n+1)}$, we have
\begin{align*}
\partial_t \| \eta^{(n+1)}_j \|_p + \tilde C_1 2^{j\gamma} \| \eta^{(n+1)}_j \|_p
\lesssim \| P_j e^{tA} (R^{\perp} e^{-tA} \eta^{(n)} \cdot \nabla e^{-tA} f^{(n+1)} ) \|_p + 
\|[P_j e^{tA},
R^{\perp} e^{-t A} f^{(n-1)}] \cdot \nabla e^{-tA} \eta^{(n+1) } \|_p.
\end{align*}
We now need a simple lemma. 
\begin{lem} \label{lem_eta}
Let $0<t\le 1$, $1\le p<\infty$, $J_1 \ge 10$. We have for any $j\ge J_1$,
\begin{align*}
&\| P_j e^{tA} ( R^{\perp} e^{-tA} \eta \cdot \nabla e^{-t A} f ) \|_p \lesssim 
2^{j(1+\frac 2p-s_0)} \| f_{[j-2,j+2]}\|_p \| \eta\|_{B^{s_0}_{p,\infty}} + 2^{j\gamma} \|\eta_{[j-2,j+3]}\|_p \cdot  \;(2^{j_1(1-\gamma+\frac 2p)}
\|f_{j_1}\|_p)_{l_{j_1}^{\infty} (j_1\ge 2J_1)}  \notag \\
&\qquad \qquad\qquad\qquad +2^{3J_1(1+\frac 2p)} \| P_{\le 3J_1} f\|_p \cdot 2^{j0+} \| \eta_{[j-2,j+3]} \|_p
+2^j \| \eta\|_{B^{s_0}_{p,\infty}} \sum_{k\ge j+4} 2^{k(\frac 2p-s_0)}\| f_{[k-2,k+2]}\|_p;
\notag \\
&\| P_j e^{tA} ( R^{\perp} e^{-tA} f \cdot \nabla e^{-tA} \eta) - R^{\perp} e^{-tA} f \cdot \nabla P_j \eta\|_p \notag \\
&\quad\lesssim 
2^{j\gamma} \| f_{>J_1} \|_{\dot B^{1+\frac 2p -\gamma}_{p,\infty} } \| \eta_{[j-2,j+3]} \|_p
+ C_{J_1} \cdot \|f_{\le J_1+10} \|_p \cdot 2^{j0+} \| \eta_{[j-2,j+3]} \|_p \notag \\
& \qquad \qquad + C_{J_1} \cdot t\| f_{\le J_1+10} \|_p \cdot 2^{j\gamma} \| \eta_{[j-2,j+3]} \|_p\notag \\
& \qquad +
\|\eta\|_{B^{s_0}_{p,\infty}} \cdot 2^{j(1+\frac 2p-s_0)} \|f_{[j-5,j+5]} \|_p+
2^{j\gamma}\|P_{\ge j+6} f\|_{B^{1+\frac 2p-\gamma}_{p,\infty}} \| \eta_j \|_p,
\end{align*}
where $C_{J_1}$ is a constant depending on $J_1$. 
\end{lem}
\begin{proof}[Proof of Lemma \ref{lem_eta}]
For the first inequality we denote $\widetilde{N_j}(g,h)= P_j e^{tA} ( R^{\perp} e^{-tA} g \cdot \nabla e^{-tA} h)$.
By frequency localization, we write
\begin{align*}
\widetilde {N_j}(\eta, f) = \widetilde{N_j}( \eta_{<j-2}, f_{[j-2,j+2]} )+ 
\widetilde{N_j}(\eta_{[j-2,j+3]}, f_{\le j+5}) +\sum_{k\ge j+4}
\widetilde{N_j}(\eta_k, f_{[k-2,k+2]} ).
\end{align*}
Clearly 
\begin{align*}
\| \widetilde{N_j}(\eta_{<j-2}, f_{[j-2,j+2]} ) \|_p \lesssim 2^j \| f_{[j-2,j+2]}\|_{p+} \| \eta_{<j-2} \|_{\infty-}
\lesssim 2^{j(1+\frac 2p-s_0)} \| f_{[j-2,j+2]} \|_p \| \eta\|_{B^{s_0}_{p,\infty}}.
\end{align*}
For the second term we split $f$ as $f= f_{>3J_1} +f_{\le 3J_1}$. Then
(below we work again with $p+$ for the $\eta$-term which give rises to $2^{j0+}$; the reason for $p+$ is
to avoid the end-point case $p=1$)
\begin{align*}
\| \widetilde{N_j} (\eta_{[j-2,j+3]}, f_{\le j+5} ) \|_p \lesssim 2^{j\gamma} \|\eta_{[j-2,j+3]}\|_p \cdot  \;(2^{j_1(1-\gamma+\frac 2p)}
\|f_{j_1}\|_p)_{l_{j_1}^{\infty} (j_1\ge 2J_1)} +2^{3J_1(1+\frac 2p)} \| P_{\le 3J_1} f\|_p \cdot 2^{j0+} \| \eta_{[j-2,j+3]} \|_p.
\end{align*}
For the diagonal piece, we have
\begin{align*}
\sum_{k\ge j+4} \| \widetilde{N_j} (\eta_k, f_{[k-2,k+2]})\|_p \lesssim 2^j \sum_{k\ge j+4} 2^{k \frac 2p} \| 
\eta_k \|_p \| f_{[k-2,k+2]} \|_p \lesssim 2^j \| \eta\|_{B^{s_0}_{p,\infty}} \sum_{k\ge j+4} 2^{k(\frac 2p-s_0)}\| f_{[k-2,k+2]}\|_p.
\end{align*}

For the second inequality, we denote
\begin{align}
N_j(f,\eta) =P_j e^{tA} (R^{\perp} e^{-tA} f \cdot \nabla e^{-tA} \eta) - R^{\perp} e^{-tA} f \cdot \nabla P_j \eta.
\end{align}
Observe that 
\begin{align}
N_j(f,\eta_{<j-2}) =P_j e^{tA} ( R^{\perp} e^{-tA} f \cdot \nabla e^{-tA} \eta_{<j-2} )
=P_je^{tA} (R^{\perp} e^{-tA} f_{[j-2,j+2]} \cdot \nabla e^{-tA} \eta_{<j-2} ). \notag 
\end{align}
Thus
\begin{align}
\|N_j(f,\eta_{<j-2} )\|_p \lesssim \| f_{[j-2,j+2]} \|_p \cdot 2^{j(1+\frac 2p-s_0)} \| 
\eta\|_{B^{s_0}_{p,\infty} }.
\end{align}
On the other hand,
\begin{align}
N_j(f, \eta_{>j+3}) =P_je^{tA} ( R^{\perp} e^{-tA} f \cdot \nabla e^{-tA} \eta_{>j+3} )
=\sum_{k\ge j+4} P_j e^{tA} ( R^{\perp} e^{-tA} f_{[k-2,k+2]} \cdot \nabla
e^{-tA} \eta_k ).
\end{align}
Thus
\begin{align}
\| N_j(f, \eta_{>j+3})\|_p \lesssim 2^j  \|\eta\|_{B^{s_0}_{p,\infty}}\sum_{k\ge j+4} 2^{k(\frac 2p-s_0)} \| f_{[k-2,k+2]} \|_p.
\end{align}
It remains to estimate $N_j(f, \eta_{[j-2,j+3]})$. We first note that
\begin{align}
\| N_j(f_{\ge j+6}, \eta_{[j-2,j+3]} )\|_p & = \| R^{\perp} e^{-tA} f_{\ge j+6} \cdot \nabla P_j \eta \|_p \notag \\
& \lesssim 2^{j\gamma}\| P_{\ge j+6} f \|_{B^{1+\frac 2p-\gamma}_{p,\infty} } \| \eta_j\|_p.
\end{align}
On the other hand,
\begin{align}
\| N_j( f_{[j-5,j+5]}, \eta_{[j-2,j+3]} ) \|_p \lesssim
2^{j(1+\frac 2p)} \| f_{[j-5, j+5]} \|_p \| \eta_{[j-2,j+3]} \|_p.
\end{align}
Finally to deal with the piece $N_j(f_{\le j-6}, \eta_{[j-2,j+3]} )$, we appeal to similar estimates
in Lemma \ref{lem_b1b} and Lemma \ref{lem_b1d}. We obtain
\begin{align}
\| N_j(f_{[j-5,j+5]}, \eta_{[j-2,j+3]} )\|_p
 &\lesssim C_{J_1} \cdot (2^{j0+} +t \cdot 2^{j\gamma} ) \| P_{\le J_1+10} f\|_p \|
 \eta_{[j-2,j+3]} \|_p  \notag \\
 & \qquad +2^{j\gamma} \| P_{>J_1} f_{<j-6} \|_{\dot B^{1+\frac 2p-\gamma}_{p,\infty} }
 \| \eta_{[j-2,j+3]} \|_p.
 \end{align}
The desired result follows.
\end{proof}

It is clear that for any $T>0$,
\begin{align}
\| \eta^{(n+1)} \|_{ C_t^0 B^{s_0}_{p,\infty} ([0,T])}
&\sim \| P_{\le 1} \eta^{(n+1) } \|_{L_t^{\infty} L_x^p([0,T])} +  (2^{js_0} \|\eta_j^{(n+1)} \|_p )_{l_t^{\infty}l_j^{\infty}
(t\in[0,T],\, j\ge 2)} \notag \\
&= \| P_{\le 1} \eta^{(n+1) } \|_{L_t^{\infty} L_x^p([0,T])} +  (2^{js_0} \|\eta_j^{(n+1)} \|_p )_{l_j^{\infty}l_t^{\infty}
(t\in[0,T],\, j\ge 2)}, \label{norm_equiv1}
\end{align}
where the implied constant (in the notation ``$\sim$'') depends only on $(s_0,p)$. 

By this simple observation, using Lemma \ref{lem_eta}, \eqref{step3_J1T1}, \eqref{initial_contra},  and choosing first $J_1$ sufficiently large
and then $T_1$ sufficiently small, we  obtain
\begin{align*}
\| \eta^{(n+1)} \|_{C_t^0 B^{s_0}_{p,\infty}([0,T_1])} \le \frac 12 \| \eta^{(n)} \|_{C_t^0 B^{s_0}_{p,\infty} ([0,T_1])}
+ C_{\theta_0} \cdot 2^{-n \sigma_0},
\end{align*}
where $C_{\theta_0}$ is a constant depending on $(\theta_0,s_0,p,\gamma,q)$ and $\sigma_0>0$ depends
on $(s_0,\gamma,p)$. This clearly yields the desired contraction in the Banach space 
$C_t^0([0,T_1], B^{s_0}_{p,\infty})$.

Step 5. Time continuity in $B^{1-\gamma+\frac 2p}_{p,q}$.  By the previous step and interpolation, we get
$f^{(n)}$ converges strongly to the limit $f$ 
also in $C_t^0 B^{s^{\prime}}_{p,1}([0,T_1])$ for any $0<s^{\prime}<1-\gamma+\frac2p$.
We still have to show $f \in C_t^0 B^{1-\gamma+\frac 2p}_{p,q}([0,T_1])$. Since $f^{(n)} \to f$  in $C_t^0 L_x^p$ we only
need to consider the high frequency part. Denote $s=1-\gamma+\frac 2p$. By using the estimates in Step 3 and strong convergence
in each dyadic frequency block, we have for any $M\ge 10$, 
\begin{align*}
\sum_{1\le j\le M} 2^{jsq} \| f_j \|_{L_t^{\infty} L_x^p([0,T_1])}^q = \lim_{n\to \infty}  \sum_{1\le j\le M} 2^{jsq} \| f_j^{(n)} 
\|^q_{L_t^{\infty} L_x^p([0,T_1])} <A_1<\infty,
\end{align*}
where $A_1>0$ is a constant independent of $M$.  Thus $ \| (f_j )\|_{l_j^q L_t^{\infty} L_x^p( j\ge 1, \, t \in [0,T_1])} <\infty$.
Since $P_{\le M} f \in C_t^0 B^s_{p,q}$ for any $M$, and 
\begin{align*}
\| P_{\le M} f - P_{\le M^{\prime}} f \|_{C_t^0 B^s_{p,q} } \lesssim  
\biggl(\sum_{M-2\le j \le M^{\prime}+2}  2^{jsq} \| f_j \|_{L_t^{\infty} L_x^p}^q \biggr)^{\frac 1q} \to 0, \quad
\text{as $M^{\prime}>M \to \infty$},
\end{align*}
we obtain $f \in C_t^0 B^s_{p,q}$. 
\begin{rem*}
An alternative argument to show time continuity is to use directly \eqref{step3_J1T1} to get time continuity at $t=0$. 
For $t>0$ one can proceed similarly as the last part of Section 3 and 
show $e^{\epsilon_0 A t} f \in L_t^{\infty} B^{1-\gamma+\frac 2p}_{p,q}$ for some $\epsilon_0>0$ small
 and use it to ``damp" the high frequencies. 
\end{rem*}

Step 6. Set $\theta(t) = e^{-t A} f(t,\cdot)$.   Clearly $\theta \in C_t^0 B^{1-\gamma+\frac 2p}_{p,q}$. Recall
$\theta_n = e^{-t A} f_n(t,\cdot)$. In view of strong convergence of $f_n$ to $f$, we have
$\theta_n \to \theta $ strongly in $C_t^0 B^{s^{\prime}}_{p,1}$ for any $0<s^{\prime} <1-\gamma+\frac 2p$. 
Since for any $0\le t_0 <t \le T_0$ we have
\begin{align*}
\theta^{(n+1)} (t) = e^{-(t-t_0) D^{\gamma}} \theta^{(n+1)}(t_0)
- \int_{t_0}^t  \nabla \cdot e^{ -(t-s) D^{\gamma}} ( R^{\perp} \theta^{(n)}   \theta^{(n+1)} )(s) ds.
\end{align*}
Taking the limit $n\to \infty$ yields 
\begin{align} \label{last_integral_form}
\theta (t) = e^{-(t-t_0) D^{\gamma}} \theta(t_0)
- \int_{t_0}^t  \nabla \cdot e^{ -(t-s) D^{\gamma}} ( R^{\perp} \theta(s)  \theta (s) ) ds.
\end{align}
It should be mentioned  that the above equality holds in the sense of $L^p_x$ and even stronger topology. It is easy to check the
absolute convergence of the integral on the RHS since
\begin{align*}
\| \nabla \cdot e^{-(t-s) D^{\gamma}} (R^{\perp} \theta(s) \theta(s) ) \|_p \lesssim (t-s)^{-1+} \|  D^{(1-\gamma)+} ( R^{\perp}
\theta(s) \theta(s) ) \|_p  
\lesssim (t-s)^{-1+} \| \theta(s) \|^2_{ B^{1-\gamma+\frac 2p}_{p,\infty}}.
\end{align*}
Thus $\theta$ is the desired local solution.  One can regard \eqref{last_integral_form} (together with
some regularity assumptions) as a variant of the usual mild solution. 
Note that $\theta_j=P_j \theta$ is smooth, and one can easy deduce from
the integral formulation \eqref{last_integral_form} the point-wise identity:
\begin{align*}
\partial_t \theta_j = - D^{\gamma} \theta_j - \nabla \cdot P_j ( \theta R^{\perp}\theta).
\end{align*}
From this one can proceed with the localized energy estimates and easily check the uniqueness of solution 
in $C_t^0 B^{1-\gamma+\frac 2p}_{p,\infty}$.
We omit the details.
\begin{rem*}
Much better uniqueness results can be obtained by exploiting the specific form of $R^{\perp}\theta$ in connection with
the $H^{-1/2}$ conservation law for non-dissipative SQG. Since this is not the focus of this work, we will not
dwell on this issue here.
\end{rem*}

\subsection*{Acknowledgements}
The author is supported in part by NSFC 12271236. 

\noindent
\textbf{Conflict of Interest}:  None.

\end{document}